\documentclass[11pt, fullpage]{amsart}

\usepackage[active]{srcltx}
\usepackage[all]{xy}
\usepackage{calc,amssymb,amsthm,amsmath}
\RequirePackage[dvipsnames,usenames]{color}

\input{kmacros3.sty}
\input{mabliautoref.sty}

\renewcommand{\O}{\mbox{$\mathcal{O}$}}

\renewcommand{\sO}{\cO}

\usepackage{fullpage}

\usepackage{amsmath}
\usepackage{amssymb}
\usepackage{amsthm}

\newcommand{\oD}{\overline{D}}
\newcommand{\oX}{\overline{X}}

\newcommand{\Hint}[1]{\vskip 3pt \indent \emph{Hint:} {#1}}

\DeclareMathOperator{\coeff}{{coeff}}
\DeclareMathOperator{\can}{{can}}
\DeclareMathOperator{\depth}{{depth}}
\DeclareMathOperator{\tr}{{tr}}

\makeatletter
\renewcommand\section{\@startsection {section}{1}{\z@}%
                                   {-3.5ex \@plus -1ex \@minus -.2ex}%
                                   {2.3ex \@plus.2ex}%
                                   {\centering\large\normalfont\bfseries\scshape}}
\renewcommand\subsubsection{\@startsection{subsubsection}{3}{\z@}%
                                     {-3.25ex\@plus -1ex \@minus -.2ex}%
                                     {1.5ex \@plus .2ex}%
                                     {\normalfont \scshape}}
\makeatother

\title{Positive Characteristic Algebraic Geometry}
\author{Zsolt Patakfalvi}
\address{EPFL SB MATHGEOM CAG\\ MA B3 444 (B\^atiment MA)\\ Station 8\\ CH-1015 Lausanne\\ Switzerland}
\email{zsolt.patakfalvi@epfl.ch}
\author{Karl Schwede}
\address{Department of Mathematics\\ The University of Utah\\ 155 S 1400 E Room 233 \\Salt Lake City, UT, 84112, USA}
\email{schwede@math.utah.edu}
\author{Kevin Tucker}
\address{Department of Mathematics, Statistics and Computer Science \\ University of Illinois at Chicago \\ 322 Science and Engineering Offices (M/C 249), 851 S. Morgan Street, Chicago, IL 60607-7045, USA}
\email{kftucker@uic.edu}
\thanks{The current version of these notes were written for the The Boot Camp for the 2015 Algebraic Geometry Summer Research Institute, which was supported by NSF DMS \# 1500652.    The first version of these notes arose from a ``Positive Characteristic Algebraic Geometry Workshop'' held at the University of Illinois at Chicago in March 2014, supported by NSF RTG grant \#1246844 and UIC.}
\thanks{The second author was partially supported by NSF FRG Grant DMS \#1265261/1501115 , NSF CAREER Grant DMS \#1252860/1501102 and a Sloan
  Fellowship.}
\thanks{The third author was partially supported by NSF grant DMS \#1419448/1602070, and a Sloan fellowship.}

\begin{document}
\bibliographystyle{skalpha}

\maketitle

\tableofcontents

\section{Introduction}

 The goal of these notes is to give a geometric introduction to recent methods utilizing the Frobenius morphism (see  \cite{SchwedeTuckerTestIdealSurvey} for more algebraic aspects) in higher dimensional algebraic geometry. These methods have been used extensively to move arguments from characteristic zero to positive characteristic. The main obstacle in doing so is usually some theorems of either analytic or p-adic origin that are known not to hold in positive characteristic: e.g., Kodaira-vanishing \cite{Raynaud_Contre_exemple_au_vanishing_theorem}, Kawamata-Viehweg vanishing \cite{XieCounterexamplesToKVVanishingOnRuledSurfaces}, some semi-positivity statements \cite[3.2]{Moret_Bailly_Familles_de_courbes_et_de_varietes_abeliennes_sur_P_1_II_exemples}, etc. The method of Frobenius splitting gives a weak replacement to  these theorems. We refer to the introduction of
\autoref{sec:global} for a list of global geometry results obtained in the past few years using the method of Frobenius split (including the existence of minimal models for threefolds).

As the goal of the notes is to serve as an introduction to the subject, we do not give  the proofs of the above mentioned global results (except one example in \autoref{sec:global}).  Instead, we aim for an introductory display of  the techniques used in the field. In particular, the notes could serve the basis of a introductory course to the subject. Further, we list many open questions, which we hope will be attacked soon by the readers of this article.   We would like to draw attention to the many exercises contained in the notes. We believe that the best approach to reading this article is to solve the exercises.

Throughout the notes we assume a basic knowledge of the language of $\bQ$-divisors (in the form in which it is widely used in birational geometry, e.g., \cite{Kollar_Mori_Birational_geometry_of_algebraic_varieties}), reflexive sheaves (c.f., \cite{Hartshorne_Generalized_divisors_on_Gorenstein_schemes,Hartshorne_Stable_reflexive_sheaves}), and Grothendieck duality for finite maps (c.f., \cite{Hartshorne_Algebraic_geometry} \cite{Hartshorne_Residues_and_duality}). A big part of the assumed background knowledge is discussed in the appendices of \cite{Schwede_Tucker_A_survey_of_test_ideals}.
\vskip 12pt
\noindent
{\it Acknowledgements:}  We would like to thank Andrew Bydlon, Eric Canton, Alberto Chiecchio, Hong R. Zong and the anonymous referee for pointing out typos in these notes.

\section{Frobenius Splittings}

The main tool we have to study a scheme $X$ with
characteristic $p > 0$ is the Frobenius endomorphism $F: X \to X$ or
$p$-th power map. Similarly, the $e$-iterated Frobenius $F^{e} : X \to
X$ is always the identity on the underlying topological space of $X$
determined by taking $p^{e}$-th powers on functions.  It is often
convenient to think of Frobenius as the morphism of sheaves of $\sO_{X}$-modules $\sO_{X} \to
F^{e}_{*}\sO_{X}$.

We say that $X$ is \emph{$F$-finite} if the Frobenius endomorphism is a
finite morphism, or equivalently $F^{e}_{*}\sO_{X}$ is coherent for
some or all $e \in \N$. One checks easily that this is the case in
essentially every geometric (or even arithmetic) situation of
interest, and \emph{$F$-finite       will be a standing assumption throughout this survey.}

\begin{exercise}
  Check that a quasi-projective variety over a perfect field is always $F$-finite.
\end{exercise}

\begin{exercise}
  When $X = \Spec(R)$ is affine and $M$ is a finitely generated
  $R$-module, justify the notation $F^{e}_{*}M$ for the module given
  by restriction of scalars for the $p^{e}$-th power map on $R$.
\end{exercise}

Note that $F^{e}_{*}$ is always an exact functor, and for any coherent
sheaf $\sM$ on $X$ we have $H^{i}(X, F^{e}_{*}\sM) = H^{i}(X, \sM)$
(though it is perhaps more accurate to write $F^{e}_{*}H^{i}(X, \sM)$
to keep track of the linearity).

\begin{theorem}[Kunz]
$X$ is regular if and only if Frobenius is flat, or equivalently (as $X$ is assumed $F$-finite)
$F^{e}_{*}\sO_{X}$ is locally free for some or all $e \in \N$.
\end{theorem}

Recall that the terms regular and smooth are interchangeable for
varieties over perfect fields.  In this case, we can think of the
iterates of Frobenius as an infinite flag
\[
\sO_{X} \subseteq F_{*}\sO_{X} \subseteq F^{2}_{*}\sO_{X} \subseteq
\cdots \subseteq F^{e}_{*}\sO_{X} \subseteq \cdots
\]
where $F^{e}_{*}\sO_{X}$ is a vector bundle of rank $p^{e(\dim X)}$.

\begin{exercise}
  Check this for $X = \bA^{n}_{k}$ where $k$ is a perfect field of
  positive characteristic.  In particular, if $S = k[x_{1}, \ldots,
  x_{n}]$ we have that $F^{e}_{*}S$ is the free $S$-module of rank
  $p^{en}$ on the basis of monomials where the exponents on each of
  the variables is strictly less than $p^{e}$.
\end{exercise}

\begin{exercise}
  Assume $k = k^{p}$ once more.  Recall that any vector bundle on
  $\bP^{1}_{k}$ is isomorphic to a direct sum of line bundles
  $\sO_{\bP^{1}}(a)$ for $a \in \Z$. Identify $F^{e}_{*} \sO_{\bP^{1}}$ as a direct sum of line bundles.
\end{exercise}

Note that the key idea for the previous exercise is to make use of the
projection formula:  for a line bundle $\sL$ and coherent sheaf $\sM$
on $X$, we have
\[
\sL \otimes_{\sO_{X}} F^{e}_{*} \sM = F^{e}_{*}\left((F^{e}{}^{*}\sL)
  \otimes_{\sO_{X}} \sM \right) = F^{e}_{*}\left( \sL^{p^{e}}
  \otimes_{\sO_{X}} \sM \right).
\]

\begin{definition}
 We say that $X$ is globally $F$-split (or simply $F$-split) if
 $\sO_{X} \to F^{e}_{*}\sO_{X}$ has a global splitting (in the
 category of $\sO_{X}$-modules).  We say that $X$ is locally $F$-split
 (or simply $F$-pure) if $\sO_{X} \to F^{e}_{*}\sO_{X}$ splits in a
 sufficiently small affine neighborhood of any point $x \in X$.
\end{definition}

\begin{example}
A smooth variety is always locally $F$-split, but may or may not be
globally $F$-split.  Our next goal is to explore what happens for
smooth projective curves, where the situation depends on the genus.
Genus zero curves are always $F$-split, genus one curves are sometimes
$F$-split, and higher genus curves are never $F$-split.
\end{example}

\begin{exercise}
  Show that locally $F$-split and globally $F$-split are the same for
  affine varieties.
\end{exercise}

\begin{lemma}
\label{lemmain2.12}
Say that $X$ is a smooth variety over $k = k^{p}$.  Then
\[
\sHom_{\sO_{X}}(F^{e}_{*}\sO_{X}, \sO_{X}) = F^{e}_{*}\sO_{X}((1-p^{e})K_{X}).
\]
\end{lemma}

\begin{proof}[Proof ingredients]
  You will need but two tricks.  First, twisting and untwisting by
  $\omega_{X}$, we have
\[
\sHom_{\sO_{X}}(F^{e}_{*}\sO_{X}, \sO_{X}) =
\sHom_{\sO_{X}}(F^{e}_{*}\sO_{X}(p^{e}K_{X}), \sO_{X}(K_{X})).
\]
The second trick is to use duality for a finite morphism, which can be
thought of as simply saying that $F^{e}_{*}(\blank)$ commutes with the
Grothendieck-duality functor $\sHom_{\sO_{X}}(\blank, \omega_{X})$.
In other words, for a coherent sheaf $\sM$ on $X$, we have
\[
\sHom_{\sO_{X}}(F^{e}_{*}\sM, \omega_{X}) =
F^{e}_{*}\sHom_{\sO_{X}}(\sM, \omega_{X}).
\]

\vspace{-.25in}
\end{proof}

\begin{exercise}
  Prove the Lemma.  Then explain how to do the same thing when $X$ is
  only normal, instead of smooth, by restricting to the smooth locus
  and using that all of the sheaves in play are reflexive.
\end{exercise}

\begin{exercise}
  If $X$ is locally $F$-split, show it is reduced and weakly normal\footnote{A reduced ring of equal characteristic $p > 0$ is called \emph{weakly normal} if for every $r \in K(R)$ the total ring of fractions of $R$ such that $r^p \in R$ also satisfies $r \in R$.}.
\end{exercise}

\begin{example}
  If $X$ is $F$-split, then it follows from Lemma~\ref{lemmain2.12} that $-nK_{X}$
  is effective for some $n \in \N$.  In particular, if $X$ has general
  type (\textit{e.g.} a smooth projective curve of genus at least two), it is never globally $F$-split.
\end{example}

\begin{lemma}
  Suppose that $X$ is a projective variety over $k = k^{p}$.  Then $X$
  is $F$-split if and only if the section ring
\[
R(X,A) = \bigoplus_{m \geq 0} H^{0}(X, \sO_{X}(mA))
\]
with respect to some (or equivalently any) ample divisor $A$ on $X$ is $F$-split.
\end{lemma}

\begin{proof}[Proof ideas]
  The main idea of the proof is straightforward enough -- twisting a
  given $F$-splitting by $\sO_{X}(mA)$ gives a splitting of
  $\sO_{X}(mA) \to F^{e}_{*}\sO_{X}(mp^{e}A)$ for any $m \geq 0$.
  Taking global sections and keeping track of the gradings as in the
  exercise below, it is then easy to see why $R(X,A)$ must be
  $F$-split.  For the other direction, we need only take $\Proj$ of a
  graded $F$-splitting.  Be careful though, if $M = \oplus_{n\in\bZ} \Gamma(X, \mathcal{M} \otimes \O_X(nA))$ then $F_* M$ is not generally the same as $\oplus_{n\in\bZ} \Gamma(X, (F_* \mathcal{M}) \otimes \O_X(nA))$.
\end{proof}

\begin{exercise}
  Suppose that $X = \Spec(R)$, where $R = \bigoplus_{m \geq 0} R_{m}$ is a
  standard-graded $k$-algebra ($R_{0} = k$ and $R = k[R_{1}]$).  For
  a $\Z$-graded $R$-module $M$, explain how to think of $F^{e}_{*}M$
  as a $\frac{1}{p^{e}}\Z$-graded $R$-module.  In particular, show
  that we can break up $F^{e}_{*}M$ into a direct sum of $\Z$-graded
  $R$-modules
\[
F^{e}_{*}M = \bigoplus_{0 \leq i < p^{e}} [F^{e}_{*}M]_{i/p^{e}
  \; \mathrm{mod} \; \Z}
\]
where $[F^{e}_{*}M]_{i/p^{e}
  \; \mathrm{mod} \; \Z} = \bigoplus_{m \in \Z} M_{i + mp^{e}}$ as an
abelian group, and the elements of $R$ act via $p^{e}$-th powers.
Finally, show that $R$ is $F$-split if and only if it has a graded
$F$-splitting, and if and only if $R_{\bm}$ is $F$-split where $\bm =
R_{+}$ is the homogeneous maximal ideal.
\end{exercise}

\begin{exercise}
  If $R \subseteq S$ is a split inclusion of rings and $S$ is
  $F$-split, then conclude $R$ is also $F$-split.  In particular,
  conclude a Veronese subring of a graded $F$-split ring is always $F$-split.
\end{exercise}

\begin{proposition}[Fedder's Criterion]
\label{prop:Fedder_s_criterion}
Let $S = k[x_{0}, \ldots, x_{n}] \supseteq I$ homogeneous, $\mathfrak{m} = \langle x_0, \ldots, x_{n} \rangle$, and $R =
S/I$.  Then $R$ is $F$-split if and only if $(I^{[p]}:I) \not\subseteq
\bm^{[p]}$, where $J^{[p]} = \langle j^{p} | j \in J \rangle$ denotes the
$p$-th bracket or Frobenius power of an ideal $J$.
\end{proposition}

\begin{example}
  $R = k[x,y,z] / \langle f = x^{3} + y^{3} + z^{3} \rangle$ is $F$-split
  if and only if $p$ is congruent to 1 modulo 3.  Indeed, one easily
  checks that these are precisely the characteristics for which
  $f^{p-1} \not\in \langle x^{p}, y^{p}, z^{p} \rangle$.
\end{example}

\begin{example}
  Suppose that $X$ is a smooth genus one curve over $k = \overline
  k$.  Recall that $X$ is said to be ordinary or have Hasse invariant
  one if the action of Frobenius on $H^{1}(X, \sO_{X})$, namely
\[
F^{*} : H^{1}(X, \sO_{X}) \to H^{1}(X, \sO_{X}),
\]
 is injective.  In this case, having chosen a base point, $X$ has precisely $p^{e}$ distinct
 $p^{e}$-torsion points for the group law on $X$.

In contrast, when the action of Frobenius on $H^{1}(X, \sO_{X})$ is
not injective (and hence identically zero), we say that $X$ is
supersingular and has Hasse invariant zero.  In this case, having
chosen a base point, the
identity element is the only $p^{e}$-torsion point for the group law.

In Theorem 4.21 of Hartshorne, it is shown that $X$ is ordinary if
and only if, when embedded as a cubic hypersurface $X = \bV(f)
\subseteq \bP^{2}$, the term $(xyz)^{p-1}$ appears with a nonzero
coefficient in $f^{p-1}$.  By Fedder's criterion, it follows that $X$
is $F$-split if and only if $X$ is ordinary.
\end{example}

\begin{exercise}
  Suppose $X$ is a smooth projective genus one curve.  First, recall Atiyah's
classification of indecomposable vector bundles on $X$.  For each $r >
0$, there exists a unique indecomposable vector bundle $E_{r}$ of rank
$r$ and degree zero with $H^{0}(X, E_{r}) \neq 0$.  Moreover, one can
construct the $E_{r}$ inductively by setting $E_{1} = \sO_{X}$ and
letting $E_{r+1}$ be the unique nontrivial extension
\[
0 \to E_{r-1} \to E_{r} \to \sO_{X} \to 0.
\]
Any other indecomposable vector bundle of rank $r$ and degree zero has
the form $E_{r} \tensor \sL$ for a uniquely determined line bundle
$\sL$ of degree zero.

Use this classification to determine the structure of $F^{e}_{*}\sO_{X}$.
If $X$ is
  ordinary and $x_{0} \in X$ is a fixed base point for the group law,
  let $x_{0}, \ldots, x_{p^{e}-1}$ be the $p^{e}$ distinct
  $p^{e}$-torsion points.  Show that $F^{e}_{*}\sO_{X} = \bigoplus_{0
    \leq i < p^{e}} \sO_{X}(x_{i}-x_{0})$.
In contrast, when $X$ is supersingular, show that $F^{e}_{*}\sO_{X}$
is the unique indecomposable vector bundle of rank $p^{e}$ with degree
zero having a nonzero global section.
\end{exercise}

\begin{proof}[Fedder's Criterion Proof Ideas]
  We sketch the two main ideas of the proof, and leave the audience to
  flesh out the details.

First, one must show that every $\phi \in \Hom_{R}(F_{*}R,R)$ is a
quotient of a $\tilde \phi \in \Hom_{S}(F_{*}S,S)$.  Conversely,
$\tilde \psi \in \Hom_{S}(F_{*}S,S)$ gives rise to a map $\psi \in
\Hom_{R}(F_{*}R,R)$ whenever $I$ is $\tilde\psi$-compatible
(\textit{i.e.} $\tilde\psi(F_{*}I) \subseteq I$).  The key point here
is that $F_{*}S$ is a free $S$-module, which allows one to lift
such (potential) splittings from $R$ up to $S$.
The second main idea of the proof is to once again make use of
duality.   We have that
\[
\Hom_{S}(F_{*}\omega_{S}, \omega_{S}) = F_{*}\Hom_{S}(\omega_{S}, \omega_{S})
= F_{*}S.
\]
In other words, identifying $S = \omega_{S}$, every element of $\Hom_{S}(F_{*}S,S)$ has the form
$\Phi_{S}(F_{*}x \cdot \blank)$ for a uniquely determined element $x
\in S$.  (One can take $\Phi_{S}$ to be the trace of Frobenius as
discussed below.)

Putting these together, if we start with a splitting $\phi$ of $R$, we
can lift it to $\tilde \phi$ on $S$.  Then $\tilde \phi (\blank) =
\Phi_{S}(F_{*}c \cdot \blank)$ for some $c \in S$.  As
$\tilde\phi(F_{*}I) \subseteq I$, one can conclude $c \in
(I^{[p]}:I)$.  As $\tilde\phi$ must be surjective since $\phi$ was a
splitting, one can also conclude $c \not\in \langle x_{0}^{p}, \ldots,
x_{n}^{p}\rangle$.  Conversely, if we can start by taking $c \in
(I^{[p]}:I) \setminus \bm^{[p]}$, one can simply take $\Phi_{S}(F_{*}c
\cdot \blank)$ and reverse the process to get a surjective map in
$\Hom_{R}(F_{*}R,R)$ -- which is just as good as having an honest F-splitting.
\end{proof}

\begin{exercise}
 Show that the dual of the Frobenius map $F^{e} : \sO_{X} \to F^{e}_{*}\sO_{X}$ of a normal variety
 takes the form $\Tr_{F^{e}} : F^{e}_{*}\omega_{X} =
 \sHom_{\sO_{X}}(F^{e}_{*}\sO_{X}, \omega_X)\to \omega_{X}$ given by
 evaluation at one.
 We call this the \emph{trace} of $F^{e}$.  Show that $\Tr_{F^{e}}$
 generates $\sHom_{\sO_{X}}(F^{e}_{*}\omega_{X}, \omega_{X})$ as an
 $F_{*}^{e}\sO_{X}$-module.

In case $X = \bA^{n} = \Spec(S = k[x_{1}, \ldots, x_{n}])$, show also that one may take the projection onto
the free summand generated by $(x_{1}\cdots x_{n})^{p-1}$ as a
$F^{e}_{*}S$-generator of $\Hom_{S}(F^{e}_{*}S,S)$.  Conclude this
agrees with the trace of $F^{e}$ up to (pre)multiplication by a unit.
\end{exercise}

\begin{exercise}
  For $X$ smooth, recall the Cartier isomorphism for the algebraic De
  Rham complex $\Omega_{X}^{\mydot}$, namely
  $\sH^{i}(F^{e}_{*}\Omega_{X}^{\mydot}) \simeq \Omega_{X}^{i}$.
  Write this down explicitly for $\bA^{n}$ in coordinates, and check
  that the induced map
\[
F^{e}_{*}\omega_{X} \to \sH^{n}(F^{e}_{*}\Omega_{X}^{\mydot}) \simeq \omega_{X}
\]
agrees (up to unit) with the trace of $F^{e}$.
\end{exercise}

\begin{exercise}
  If $A$ is a finitely generated $\Z$-algebra and $\bm$ is a
  maximal ideal of $A$, show that $A/\bm$ is a finite field.
\end{exercise}

\begin{definition}
  Suppose that $X$ is a complex algebraic variety.  One can form a
  finitely generated $\Z$-algebra domain $A$ and an arithmetic family
  $\sX \to \Spec(A)$ so that $\sX_{0} \otimes_{\Frac(A)} \bC = X$.  We
 call $\sX$  a model of $X$ over $\Spec(A)$.  $X$ is said to have
 dense (local or global, respectively) $F$-split type if $\sX_{\bm}$
 is (local or global, respectively) $F$-split for a dense set of
 closed points $\bm \in \Spec(A)$.
\end{definition}

\begin{conjecture}
  If $X$ is a complex Abelian variety (or more generally log
  Calabi-Yau), then $X$ has dense $F$-split type.
\end{conjecture}

\begin{proposition}
  Suppose that $X$ is a globally $F$-split projective variety.
  \begin{enumerate}
  \item If $\sL$ is any line bundle so that $H^{i}(X, \sL^{\otimes n})
    = 0$ for $n \gg 0$, then $H^{i}(X, \sL) = 0$.
  \item If $A$ is an ample divisor on $X$, then $H^{i}(X, \sO_{X}(A))
    = 0$ for all $i > 0$.
  \item Suppose $X$ is Cohen-Macaulay.  If $A$ is an ample divisor on $X$ and $i < \dim X$, then
    $H^{i}(X, \sO_{X}(-A))=0$.
  \item Suppose $X$ is normal.  If $A$ is an ample divisor on $X$ then $H^{i}(X,
    \omega_{X}(A)) = 0$ for $i > 0$.
  \end{enumerate}
\end{proposition}

\begin{definition}
  If $D$ is an effective Cartier divisor on a normal variety $X$, we say that $X$ is
  $e$-$F$-split along $D$ if $\sO_{X} \to F^{e}_{*}\sO_{X} \subseteq
  F^{e}_{*}\sO_{X}(D)$ splits (globally).
\end{definition}

\begin{lemma}
  If $X$ is $e$-$F$-split along an ample divisor $A$ and $\sL$ is nef,
  then $H^{i}(X, \sL) = 0$ for all $i > 0$.
\end{lemma}

\begin{exercise}
  Suppose $X$ is normal. If $X$ is $e$-$F$-split along an ample effective divisor $A$, then $-K_{X}$ is big.
\end{exercise}

\begin{definition}
  We say that a normal variety $X$ is globally $F$-regular if, for every effective
  Cartier divisor $D$, there is some $e > 0$ so that $X$ is
  $e$-$F$-split along $D$.  We say that $X$ is locally $F$-regular (or strongly $F$-regular) if,
  for every effective Cartier divisor $D$, there is some $e > 0$ so that
  $\sO_{X} \to F^{e}_{*} \sO_{X}(D)$ is split in a neighborhood of any point.
\end{definition}

\begin{exercise}
  Suppose that $X = \Spec(R)$ is affine.  Then $X$ is globally
  $F$-regular if and only if $X$ is locally $F$-regular if and only
  if, for all $0 \neq c \in R$, there exists $e > 0$ such that $R \to
  F^{e}_{*}R$ given by $1 \mapsto F^{e}_{*}c$ is split.
\end{exercise}

\begin{exercise}
  Show that a regular ring is always strongly $F$-regular, and that a
  strongly $F$-regular ring is a Cohen-Macaulay normal domain.
\end{exercise}

\begin{exercise}
  If $X$ is a smooth projective curve, then $X$ is globally
  $F$-regular if and only if $X$ has genus zero.
\end{exercise}

\begin{exercise}
  If $X$ is globally $F$-regular and $D$ is big and nef, then
  $H^{i}(X, \omega_{X}(D)) = 0$ for all $i > 0$.
\end{exercise}
\section{Trace of Frobenius and Global Sections}
\label{sec:TraceOfFrob}
Suppose $X$ is a normal quasi-projective variety over a perfect field $k$ of characteristic $p > 0$.  We just read about Frobenius splittings $F^e_* \cO_X \to \cO_X$.  We will generalize this notion.  Fix $\sL$ to be a line bundle and suppose that we have a non-zero $\cO_X$-linear map $\phi : F^e_* \sL \to \cO_X$.

Twist $\phi$ by $\sL$ and apply $F^e_*$ and by the projection formula we get:
\[
F^{2e}_* (\sL^{1+p^e}) \to F^e_* \sL
\]
We compose this back with $\phi$ to obtain:
\[
\phi^2 : F^{2e}_* (\sL^{1+p^e}) \to \cO_X.
\]
This is an abuse of notation of course.

By repeating this operation we get
\[
\phi^n : F^{ne}_* (\sL^{1 + p^e + \ldots + p^{(n-1)e}}) \to \cO_X.
\]

\begin{theorem}\cite{HartshorneSpeiserLocalCohomologyInCharacteristicP,LyubeznikFModulesApplicationsToLocalCohomology,Gabber.tStruc}
\label{thm.HSL}
With notation as above, there exists $n_0$ such that $\Image(\phi^n) = \Image(\phi^{n+1})$ as sheaves for all $n \geq n_0$.
\end{theorem}
\begin{proof}[Sketch of a proof]
Let $U_n$ be the locus where $\Image(\phi^n)$ and $\Image(\phi^{n+1})$ coincide.  It is easy to see that the $U_n$ are ascending open sets and so they stabilize by Noetherian induction.  By localizing at the generic point of $X \setminus U_n$ for $n \gg 0$ one can assume that $X$ is the spectrum of a local ring with closed point $x$ and that $U_n = X \setminus x$.  The trick is then to show that $\bm_x^{l} \cdot \Image(\phi^m) \subseteq \Image(\phi^n)$ for some $l > 0$ and all $n \geq m \gg 0$.  Finding this $l$ is left as an exercise to the reader.
\end{proof}

These images measure the singularities of the pair $(X, \phi)$.  We'll see the relevance of $\phi$ momentarily.

\begin{definition}
We define $\sigma(X, \phi) = \Image(\phi^n)$ for $n \gg 0$.  We say the $(X, \phi)$ is \emph{$F$-pure} if $\sigma(X, \phi) = \cO_X$.
\end{definition}

\begin{exercise}
Suppose that $\phi : F^e_* \sL \to \O_X$ is surjective.  Prove that $(X, \phi)$ is $F$-pure.
\end{exercise}

Let us now describe the significance of $\phi$ in terms of divisor pairs that you might be more familiar with.

\begin{proposition}
\label{prop.MapDivisorCorrespondence}
There is a bijection of sets
\[
\left\{ \begin{array}{c}\text{effective $\bQ$-divisors $\Delta$ such that}\\ \text{ $(p^e - 1)(K_X + \Delta)$ is Cartier}\end{array} \right\}\Big/\sim\;\; \longleftrightarrow \left\{ \begin{array}{c}\text{line bundles $\sL$ and non-zero}\\ \text{$\cO_X$-linear maps $\phi : F^e_* \sL \to \cO_X$}\end{array} \right\}
\]
where the equivalence relation on the left declares two maps to be equivalent if they agree up to multiplication by a unit from $\Gamma(X, \O_X)$.
\end{proposition}
\begin{proof}[Sketch]
There is always a map $F^e_* \O_X(K_X) \to \O_X(K_X)$, hence by the projection formula (and reflexification) we get a map $F^e_* \O_X( (1-p^e)K_X) \to \O_X$.  If $\Delta$ is effective, we then get the composition
\[
F^e_* \O_X( (1-p^e)(K_X+\Delta)) \subseteq F^e_* \O_X( (1-p^e)K_X) \to \O_X.
\]
Thus we obtain the $\Rightarrow$ direction for $\sL = \O_X( (1-p^e)(K_X + \Delta))$.  The reverse direction follows easily from the isomorphism $F^e_* (\sL^{-1}\otimes \O_X( (1-p^e)K_X)) \cong \sHom_{\O_X}(F^e_* \sL, \O_X)$ coming from Grothendieck duality for finite maps.  Given this, a map $\phi$ determines a global section of $\sL^{-1}\otimes \O_X( (1-p^e)K_X)$ which determines an effective Weil divisor $D_{\phi}$.  Set $\Delta = {1 \over p^e - 1} D_{\phi}$.
\end{proof}

Given any effective $\bQ$-divisor $\Delta$ such that $K_X + \Delta$ is $\bQ$-Cartier with index not divisible by the characteristic $p$, one can find $e > 0$ such that $(p^e - 1)(K_X + \Delta)$ is Cartier and hence find a corresponding $\phi_{e, \Delta}$.

\begin{exercise}
If $K_X + \Delta$ is $\bQ$-Cartier with index not divisible by $p > 0$, show that there exists an $e > 0$ such that $(p^e - 1)(K_X + \Delta)$ is Cartier.
\end{exercise}

In particular, this is exactly how the above was presented in the mini-lecture.  Given such a $\Delta$ one can form $\phi_{e, \Delta} : F^e_* \O_X\big( (1-p^e)(K_X + \Delta) \big) \to \O_X$ for sufficiently divisible $e$.  \autoref{thm.HSL} says that  $\sigma(X, \Delta)$ has stable image for $e \gg 0$ and sufficiently divisible.  Translated this way consider the following exercise:

\begin{exercise}
With notation as above, suppose that $d > e$ are integers such that $(p^i-1)(K_X + \Delta)$ is Cartier for $i = d,e$.  Show that
\[
\phi_{d, \Delta} : F^d_* \O_X\big( (1-p^d)(K_X + \Delta) \big) \to \O_X
\]
factors through
\[
\phi_{e, \Delta} : F^e_* \O_X\big( (1-p^e)(K_X + \Delta) \big) \to \O_X.
\]
Further show that for any $b, c$ satisfying the condition on $i$ above, the composition
\[
\begin{array}{rl}
& F^{b+c}_* \O_X( (1 - p^b)(K_X + \Delta) + p^b(1-p^c)(K_X + \Delta)) \\
\xrightarrow{F^b_* \big(\phi_{b, \Delta} \tensor \O_X( (1-p^c)(K_X + \Delta) ) \big)} & F^c_* \O_X\big( (1-p^c)(K_X + \Delta) \big)  \\
\xrightarrow{\phi_{c, \Delta}} & \O_X
\end{array}
\]
is simply $\phi_{b+c, \Delta}$.
\end{exercise}

\begin{definition}
Given a $\bQ$-divisor $\Delta$ such that $K_X + \Delta$ is $\bQ$-Cartier with index not divisible by $p > 0$, we say that the pair $(X, \Delta)$ is $F$-pure if $(X, \phi_{e, \Delta})$ is $F$-pure for some (equivalently any) $\phi_{e, \Delta}$ corresponding to $\Delta$ as in \autoref{prop.MapDivisorCorrespondence}.
\end{definition}

\begin{exercise}
Suppose that $X$ is an affine cone over a smooth hypersurface $Y$ in $\bP^n$.  Further suppose that $H^0(Y, F^e_* \omega_Y(p^e n)) \to H^0(Y, \omega_Y(n))$ surjects for all $n \geq 0$ and all $e \geq 0$.  Let $\pi : X' \to X$ be the blowup of $X$ at the cone point (a resolution of singularities).  Show that $\sigma(X, 0) = \pi_* \O_{X'}(K_{X'/X} + E)$ where $E \cong Y$ is the exceptional divisor.
\Hint{Recall that $\omega_X = \oplus_{i \in \bZ} H^0(Y, \omega_Y(i))$ which is just $\O_X$ with a shift in grading (since $X$ is a hypersurface).}
\end{exercise}

\begin{remark}
Suppose we now consider a general situation of a pair $(X, \Delta)$ with $K_X + \Delta$ $\bQ$-Cartier.
Philosophically, one should expect that $\sigma(X, \Delta)$ to correspond with something like the following ideal sheaf $\pi_* \O_{X'}(\lceil K_{X'} + \pi^*(K_X + \Delta) + \varepsilon E\rceil)$ where $\pi : X' \to X$ is a log resolution and $E = \Supp(\pi^{-1}_* \Delta) \cup \text{Exc}_{\pi}$.  See \cite{FujinoSchwedeTakagiSupplements,TakagiAdjointIdealsAndACorrespondence,BhattSchwedeTakagiWeakOrdinarity}.
\end{remark}

We can globalize this process.

\begin{definition}
Suppose that $X$ is a proper variety, $\Delta \geq 0$ is a $\bQ$-divisor such that $K_X+ \Delta$ is $\bQ$-Cartier and $M$ is a line bundle.  We define
\[
S^0(X, \sigma(X, \Delta) \otimes M)
\]
to be the image
\[
\Image\Big( H^0\big(X, F^e_* \big( \O_X( (1-p^e)(K_X + \Delta)) \otimes M^{p^e}\big) \big) \to H^0(X,  \sigma(X, \Delta) \otimes M \Big) \subseteq H^0(X, M) \Big)
\]
for some $e \gg 0$.

We make a similar definition even without the presence of $\Delta$.  We define $S^0(X, \omega_X \otimes M)$
to be the image
\[
\Image\Big( H^0\big(X, F^e_* \big( \omega_X \otimes M^{p^e}\big) \big) \to H^0(X, \omega_X \otimes M) \Big)
\]
for $e \gg 0$.
Here the map is induced by $F^e_* \omega_X \cong F^e_* \O_X(K_X) \to \O_X(K_X) = \omega_X$ mentioned earlier.
\end{definition}

\begin{exercise}
Suppose that $X$ is Frobenius split by some map $\phi_{e, \Delta} : F^e_* \O_X \to \O_X$ which corresponds to a Weil divisor $(p^e - 1)\Delta = D \sim (1-p^e)K_X$ as above.  Show that $S^0(X, \sigma(X, {1 \over p^e-1} D) \otimes M) = H^0(X, M)$ for any line bundle $M$.
\end{exercise}

\begin{question}
Find a way to compute $S^0(X, \omega_X \otimes M)$ for $M$ that is not necessarily ample.  In particular, how do you know when you have chosen $e > 0$ large enough?
\end{question}

Note that this object depends on the map used to construct $\sigma(X, \Delta)$ and not \emph{simply} on the sheaf $\sigma(X, \Delta)$.  Technically, we are relying on the structure of $\sigma(X, \Delta)$ as a Cartier module \cite{BlickleBockleCartierModulesFiniteness,BlickleSchwedeSurveyPMinusE}.  Likewise $S^0(X, \omega_X \otimes M)$ depends on the map $F^e_* \omega_X \to \omega_X$ (fortunately, that map is canonically determined).

\begin{remark}
The letter $S^0$ should be thought of as a capitalized or globalized $\sigma$.
\end{remark}

The sections $S^0(X, \sigma(X, \Delta))$ are important because they behave as though Kodaira (or Kawamata-Viehweg) vanishing were true.  Perhaps this should not be surprising, indeed in order to construct counter-examples to Kodaira vanishing in characteristic $p > 0$ \cite{raynaud_contre-exemple_1978} one typically first finds a curve $X$ and line bundle $M$ such that $S^0(X, \omega_X \otimes M) \neq H^0(X, \omega_X \otimes M)$.  This failure of equality is the key ingredient in proving that Kodaira vanishing fails.

Consider the following example.
\begin{example}
\label{ex.ExtendingSections}
Suppose that $X$ is a normal projective variety, $M$ is a Cartier divisor, $\Delta$ is a $\bQ$-divisor and $D$ is a normal Cartier divisor such that $D$ and $\Delta$ have no common components.  We also assume that $M - K_X - \Delta-D$ is an ample $\bQ$-divisor.  Then $S^0(X, \sigma(X, \Delta+D) \otimes \O_X(M) )$ surjects onto $S^0(D, \sigma(D, \Delta|_D) \otimes \O_D(M|_D))$.

This is actually quite easy.  First consider the following commutative diagram:
\[
{\small
\xymatrix@C=12pt{
F^e_* \O_X( (1-p^e)(K_X + \Delta + D) - D) \ar[d]_{\alpha} \ar[r] & F^e_* \O_X( (1-p^e)(K_X + \Delta + D)) \ar[d]_{\beta} \ar[r] & F^e_* \O_D( (1-p^e)(K_X + \Delta + D)|_D ) \ar[d]_{\gamma} \\
\O_X(-D) \ar@{^{(}->}[r] &  \O_X \ar@{->>}[r]_{\rho} & \O_D
}
}
\]
Note that the upper right term is $F^e_* \O_D( (1-p^e)(K_D + \Delta|_D) )$ and it is not difficult\footnote{Well, maybe it's a little difficult, it is an exercise below.} to see that the map $\gamma$ is the one whose image on global sections is the one used to define $\sigma(D, \Delta|_D)$.

We then twist by $M$, use the projection formula and take cohomology to obtain:
\begin{equation}
\label{eq.LiftingDiagram}
{\scriptsize
\xymatrix@C=15pt{
\ldots \ar[d]_{\alpha} \ar[r] & H^0(X, F^e_* \O_X( (1-p^e)(K_X + \Delta + D) + p^e M)  \ar[d]_{\beta} \ar[r]^{\delta} & H^0(X, F^e_* \O_D( (1-p^e)(K_D + \Delta|_D)+p^e M|_D )  )\ar[d]_{\gamma} \\
H^0(X, \O_X(-D + M) ) \ar@{^{(}->}[r] &  H^0(X,\O_X(M)) \ar[r]_{\rho} & H^0(D,\O_D(M|_D))
}
}
\end{equation}
The map labeled $\rho$ need not be surjective.  However note that
\[
H^1(X, \O_X( (1-p^e)(K_X + \Delta + D) - D + p^e M) =  H^1(X, \O_X(M - D) \otimes \O_X( (p^e - 1)(M - K_X - \Delta - D)))
\]
which vanishes due to Serre vanishing for $e \gg 0$.  It follows that $\delta$ is surjective.  The result follows from an easy diagram chase since the image of $\beta$ is $S^0(X, \sigma(X, \Delta+D) \otimes \O_X(M) )$ and the image of $\gamma$ is $S^0(D, \sigma(D, \Delta|_D) \otimes \O_D(M|_D))$.
\end{example}

There are some details in the above example that I leave as an exercise.

\begin{exercise}
Show that diagram \autoref{ex.ExtendingSections} commutes at least in the case that $X$ is smooth or at least factorial.  To do this, start with the following diagram:
\[
\xymatrix@R=16pt{
0 \ar[r] & F^e_* \O_X(-D) \ar[r] & F^e_* \O_X \ar[r] & F^e_* \O_D \ar[r] & 0\\
 & F^e_* \O_X(-p^e D) \ar@{^{(}->}[u] & \\
0 \ar[r] & \O_X(-D) \ar[u]_{F^e} \ar[r] & \O_X \ar[r] \ar[uu]_{F^e} & \O_D \ar[uu]_{F^e} \ar[r] & 0
}
\]
Apply $\sHom_{\O_X}(\blank, \omega_X)$ and then twist everything by $\O_X(-K_X - D)$.  Additionally twist the $F^e_*$-row by $F^e_* \O_X( (1-p^e)\Delta)$ and note we still have a map.
\end{exercise}

\begin{exercise}
In general, a divisor $\Delta$ induces a map $F^e_* \O_X( (1-p^e)(K_X + \Delta + D)) \to \O_X$ which induces a map $F^e_* \O_D((1-p^e)(K_X + \Delta + D)|_D) \to \O_D$ by restriction.  This map then induces a divisor $\Delta_D$ called the \emph{$F$-different}.  It was recently shown by O.~Das that the $F$-different coincides with the usual notion of \emph{different} in full generality, see \cite{DasOnStronglyFregularInversion}.  Verify carefully in the case of the previous exercise\footnote{or at least when $D$ is Cartier and $(p^e-1)\Delta$ is Cartier} that the $F$-different is exactly $\Delta|_D$.
\end{exercise}

\subsection{Examples and computations of $S^0$}

\begin{exercise}
Use a diagram similar to the one in the example above to show the following.  Suppose that $C$ is a curve and $L$ is ample.  Show that $S^0(C, \omega_C \otimes L)$ defines a base point free linear system if $\deg L \geq 2$ and that $S^0(C, \omega_C \otimes L)$ defines an embedding if $\deg L \geq 3$.
\Hint{To show that it defines a base point free linear system, choose a point $Q \in C$ and consider $\omega_C \otimes L \to (\omega_C \otimes L)/(\omega_C \otimes L \otimes \O_C(-Q))$.}
\end{exercise}

One can also ask when $S^0 \neq 0$, \cite{TanakaTraceOfFrobenius} showed this when $\deg L \geq 1$ and $C \neq \bP^1$.

\begin{question}
Suppose $X$ is a surface, or maybe a ruled surface to start, and $L$ is a line bundle on $X$.  What conditions on $L$ imply that $S^0(X, \omega_X \otimes L)$
\begin{itemize}
\item{} is nonzero?
\item{} defines a base point free linear system?
\item{} induces an embedding $X \subseteq \bP^n$?
\end{itemize}
What about other special varieties (smooth cubics, elliptic surfaces?)
\end{question}

\begin{exercise}
Suppose that $(X, \Delta)$ is as above and that $M$ is a line bundle such that $M - K_X - \Delta$ is ample.  Further suppose that $N$ is an ample divisor such that $\O_X(N)$ is globally generated by $S^0$.  Use Castelnuovo-Mumford regularity to prove that $\sigma(X, \Delta) \tensor \O_X(M + (\dim X)N)$ is globally generated.
\Hint{Consider the map $F^e_* \O_X( (1-p^e)(K_X + \Delta) + (p^e - 1)M + M + p^e (\dim X) N) \to \O_X(M + (\dim X)N)$ and show that the source is globally generated as an $\O_X$-module using Castelnuovo-Mumford regularity.  See \cite{Keeler_Fujita_s_conjecture_and_Frobenius_amplitude,Schwede_A_canonical_linear_system}.}
\end{exercise}

\begin{question}[Hard]
With notation as above, find a constant $c(d)$, depending only on $d = \dim X$ such that if $X$ is a smooth projective variety and $L$ is ample and $M_{c(d)} = K_X + L$, then $S^0(X, \sigma(X, 0) \otimes M)$ yields a base point free linear system.
\end{question}

Later in \autoref{sec:SeshadriEtc}, we will describe several other ways to produce sections in $S^0$.

\section{$F$-singularities versus Singularities of the MMP}

In this section we discuss the relationship between both global and local notions of $F$-singularity theory and of Minimal Model Program. The loose picture is summarized below:
\vspace{5pt}
\begin{center}
 \begin{tabular}{l|cl}
globally $F$-regular & $\sim$ & log-Fano \\
strongly $F$-regular & $\sim$ & Kawamata log terminal \\
globally $F$-split & $\sim$ & log-Calabi-Yau \\

sharply $F$-pure & $\sim$ & log canonical \\
\hline
work in proofs & & interesting geometry
\end{tabular}
\end{center}
\vspace{5pt}
Unfortunately there is a slight discrepancy between the left side, which works for proofs, and the right side, which is geometrically interesting. So, whenever one would like to prove something geometrically interesting has to deal with this discrepancy, the extent of which is the main topic of this section.

We work over a perfect base-field $k$ of characteristic $p>0$. Unless otherwise stated, a pair $(X, \Delta)$ denotes a normal, essentially finite type scheme $X$ over $k$ endowed with an effective $\bQ$-divisor $\Delta$.

First, recall the following definition.

\begin{definition}
A pair $(X, \Delta)$ is globally  $F$-split if  the natural map $\sO_X \to F_*^e \sO_X (\lceil (p^e -1) \Delta \rceil) $ admits a splitting for some integer $e>0$. A pair $(X, \Delta)$ is sharply $F$-pure if it has an affine open cover by globally $F$-regular pairs.
\end{definition}

Note that $\lceil (p^e -1) \Delta \rceil$ is a Weil divisor. As we observed before, there is an associated sheaf $\sO_X(\lceil (p^e -1) \Delta \rceil)$ defined with sections
\begin{equation*}
\sO_X(\lceil (p^e -1) \Delta \rceil)(U):=\{ f \in K(X) |  (f)|_U + \lceil (p^e -1) \Delta \rceil|_U \geq 0 \}.
\end{equation*}
Further, this sheaf is $S_2$ or equivalently reflexive (see \cite{Hartshorne_Generalized_divisors_on_Gorenstein_schemes,Hartshorne_Stable_reflexive_sheaves} for generalities on reflexive sheaves). In particular its behavior is determined in codimension one.

Now we define the following related notions that, as we will see, are related to globally $F$-split and sharply $F$-pure as log canonical is related to Kawamata log terminal.

\begin{definition}
A pair $(X, \Delta)$ is globally  $F$-regular if for every effective $\bZ$-Weil divisor $D$, the natural map $\sO_X \to F_*^e \sO_X (\lceil (p^e -1) \Delta \rceil + D) $ admits a splitting for some integer $e>0$. A pair $(X, \Delta)$ is strongly $F$-regular if it has an affine open cover by globally $F$-regular pairs.
\end{definition}

\begin{remark}
Strongly $F$-regular is a local condition, i.e., $(X, \Delta)$ is strongly $F$-regular, if and only if for each $x \in X$, $(\Spec \sO_{X,x}, \Delta_x)$ is strongly $F$-regular \cite[Exercise 3.10]{Schwede_Tucker_A_survey_of_test_ideals}. Also, then every affine open set and every local ring of  a strongly $F$-regular pair is strongly $F$-regular.
\end{remark}

\begin{exercise}
\label{exc:explicit_gfr_general}
Let $0 \leq  a_1, a_2, a_3 \leq 1$ be rational numbers. Show that $\left(\bP^1, a_1 P_1 + a_2 P_2 + a_3 P_3 \right)$ is globally $F$-regular (resp. globally $F$-split) if and only if for some integer $e>0$, \[x^{\lceil (p^e -1) a_1 \rceil} y^{\lceil (p^e -1) a_2 \rceil} ( x+y)^{\lceil (p^e -1) a_3 \rceil}\] has a non-zero monomial $x^i y^j$ with $0 \leq i, j < p^e -1$ (resp. $0 \leq i, j \leq p^e -1$) for some $e >0$.

\Hint{Show that we can assume that $P_1=0$, $P_2=1$ and $P_3 = \infty$. Describe then explicitly the maps
\begin{equation}
\label{eq:globall_F_regular_exercise}
H^0 \left(\bP^1, \lfloor (1-p^e)(K_X + a_1 P_1 + a_2 P_2 + a_3 P_3) -D \rfloor \right) \to H^0\left(\bP^1, (1-p^e)K_X  \right)  \to H^0 \left(\bP^1, \sO_{\bP^1} \right)
\end{equation}
( resp., $H^0 \left(\bP^1, \lfloor (1-p^e)(K_X + a_1 P_1 + a_2 P_2 + a_3 P_3)  \rfloor \right) \to H^0 \left(\bP^1, (1-p^e)K_X  \right) \to H^0 \left(\bP^1, \sO_{\bP^1} \right)$). Concluding the globally $F$-split case is then easy. In the globally $F$-regular case one has to deal with $D$, since if $\deg D >1$ then the existence of a non-zero monomial $x^i y^j$ with $0 \leq i, j < p^e -1$ does not guarantee that the composition of the maps in \autoref{eq:globall_F_regular_exercise} is surjective. The idea here is to show that in this situation \autoref{eq:globall_F_regular_exercise} becomes  surjective if one replaces $e$ by $ne$ for any $n \gg 0$. For this, show that if $x^i y^j$ is a non-zero monomial of $x^{\lceil (p^e -1) a_1 \rceil} y^{\lceil (p^e -1) a_2 \rceil} ( x+y)^{\lceil (p^e -1) a_3 \rceil}$ such that $0 \leq i, j < p^e -1$ and $i$ and $j$ are minimal with these properties (or rather precisely $i$ is minimal and $j$ is minimal amongst having that $i$), then $x^{i \frac{p^{ne} -1}{p^e -1}} y^{j
\frac{p^{ne} -1}{p^e -1}}$ is a non-zero monomial of $\left(x^{\lceil (p^e -1) a_1 \rceil} y^{\lceil (p^e -1) a_2 \rceil} ( x+y)^{\lceil (p^e -1) a_3 \rceil} \right)^{\frac{p^{ne} -1}{p^e -1}}$. Now, using that $\lceil  (p^{ne} -1) a_i \rceil \leq \lceil (p^e -1) a_i \rceil \frac{p^{ne} -1}{p^e -1}$ we obtain a non-zero coefficient $x^{i'} y^{j'}$ of $x^{\lceil (p^{ne} -1) a_1 \rceil} y^{\lceil (p^{ne} -1) a_2 \rceil} ( x+y)^{\lceil (p^{ne} -1) a_3 \rceil}$, such that $i' + \frac{p^{ne}-1}{p^e -1},j' + \frac{p^{ne}-1}{p^e -1} \leq p^{ne}-1$. Conclude then the surjectivity of \autoref{eq:globall_F_regular_exercise} for $e$ replaced by $ne$ when $n \gg 0$.}
\end{exercise}

\begin{exercise}
\label{exc:explicit_gfr}
Show that $\left( \bP^1 , \frac{1}{2} P_1 + \frac{1}{2} P_2 + \frac{n-1}{n} P_3 \right)$ is globally $F$-regular if and only if $p \neq 2$.
\Hint{Apply the previous exercise. You can make your life easier by permuting the points.}
\end{exercise}

%

\begin{exercise}
\label{exc:globally_F_regular_all_e}
Show that if $X$ is globally $F$-regular, then the splitting condition of the definition happens for all $e \gg 0$ (where the bound  depends on $D$). That is, for each effective divisor $D$, there is an integer $e_D$, such that $\sO_X \to F_*^e \sO_X (\lceil (p^e -1) \Delta \rceil + D) $ admits a splitting for all integers $e \geq e_D$.
\Hint{The splitting $\phi_e$ for $e$ gives a splitting for $2e$ by restricting $\phi_e \circ F^e_* ( \phi_e)$ from \\ $F_*^{2e} \sO_X \left(\lceil (p^{e} -1) \Delta \rceil + p^e \lceil (p^{e} -1) \Delta \rceil +  \frac{p^{2e}-1}{p^e-1}D \right)$ to  $F_*^{2e} \sO_X (\lceil (p^{2e} -1) \Delta \rceil + D)$ (it has to be used somewhere that $F_*$ of a reflexive sheaf is reflexive, which can be deduced from \cite[1.12]{Hartshorne_Generalized_divisors_on_Gorenstein_schemes}). Then by induction one obtains by a similar method  a splitting  for every divisible enough $e$. To obtain it for every big enough $e$, replace $D$ by $D':= \lceil \Delta \rceil + D$.  The same  argument then yields a splitting $\psi_n$ of $\sO_X \to F_*^{ne} \sO_X \left(\lceil  (p^{ne} -1) \Delta \rceil +  \frac{p^{ne}-1}{p^e-1}D' \right)$. Finally note that by the choice of $D'$, for every $r \geq e$, there is a natural inclusion $F^r_* \sO_X\left(\lceil  (p^{r} -1) \Delta \rceil +  D \right)$ into $F^{me}_* \sO_X \left(\lceil  (p^{me}
-1) \Delta \rceil +  \frac{p^{me}-1 }{p^e-1}D' \right)$, where $m = \left\lceil \frac{r}{e} \right\rceil$. Restricting $\psi_m$ via this inclusion yields the desired splitting for every $r \geq e$.}
\end{exercise}

\begin{exercise}
Show that in the definition of globally $F$-regular we could have assumed $D$ to be an ample Cartier divisor as long as $X$ is projective.
\Hint{Use a similar argument to that in \autoref{exc:globally_F_regular_all_e}: the key is that by replacing $e$ by $ne$ we may replace $D$ by $\frac{p^{ne}-1}{p^e-1} D$.}
\end{exercise}

\begin{exercise}
\label{ex:Hara_gfr_def}
Show that $(X, D)$ is globally $F$-regular if and only if for every effective $\bZ$-divisor $D$, the natural map $\sO_X \to F_*^e \sO_X ( \lfloor p^e  \Delta \rfloor + D) $ admits a splitting for some integer $e>0$.
\Hint{Since we have splitting for all $D$ we may swallow into $D$ any difference in the expression that is bounded as $e$ grows. Now note that
\begin{equation*}
-\Delta  \leq \lceil (p^e -1) \Delta \rceil - \lfloor p^e  \Delta \rceil \leq 2 \lceil \Delta \rceil - \Delta .
\end{equation*}
}
\end{exercise}

\begin{remark}
The analogous statement to \autoref{ex:Hara_gfr_def} does not hold for globally $F$-split pairs.
\end{remark}

\begin{exercise}
\label{exc:dual_form_of_global_F_split_F_regular}
Show that $(X, \Delta)$ is globally $F$-regular if and only if for all effective divisors $D$ the natural map
\begin{equation*}
H^0 \Big(X, F^e_* \sO_X \big(\lfloor (1-p^e) (K_X + \Delta) \rfloor -D\big)\Big) \to H^0(X, \sO_X)
\end{equation*}
is surjective  for some $e>0$. Recall that the natural map $F^e_* \sO_X (\lfloor (1-p^e) (K_X + \Delta) \rfloor -D) \to \sO_X$ is constructed as the composition
\begin{equation*}
F^e_* \sO_X (\lfloor (1-p^e) (K_X + \Delta) \rfloor -D) \to F^e_* \sO_X ( (1-p^e) K_X ) \to \sO_X,
\end{equation*}
 where the latter map is the twist of the Grothendieck trace map by $\sO_X(-K_X)$ using the projection formula (and that all the sheaves involved are $S_2$).

Similarly show that  $(X, \Delta)$ is globally $F$-regular if and only if for all effective divisors $D$ the natural map
\begin{equation*}
H^0 \Big(X, F^e_* \sO_X \big( \lceil (1-p^e) K_X -  p^e \Delta \rceil -D\big)\Big) \to H^0(X, \sO_X)
\end{equation*}
is surjective  for some $e>0$.
\Hint{use duality theory, that is, apply $R \Hom( \_, \omega_X)$ to $\sO_X \to F_*^e \sO_X (\lceil (p^e -1) \Delta \rceil + D) $ and then apply also $\_ \otimes \sO_X(-K_X)$. For the second claim use \autoref{ex:Hara_gfr_def}.}
\end{exercise}

\begin{definition}
A pair  $(X, \Delta)$ is klt (Kawamata log terminal) if $K_X + \Delta$ is $\bQ$-Cartier and for all birational morphisms from normal varieties $f : Y \to X$, $\lceil K_Y - f^*(K_X + \Delta) \rceil \geq 0$, where $K_Y$ and $K_X$ are chosen such that $f_* K_Y = K_X$.
\end{definition}

\begin{definition}
A pair  $(X, \Delta)$ is lc (log canonical) if for all birational morphisms from normal varieties $f : Y \to X$, $ K_Y - f^*(K_X + \Delta) \geq -E $ for some reduced divisor $E$, where $K_Y$ and $K_X$ are chosen such that $f_* K_Y = K_X$.
\end{definition}

\begin{remark}
If $(X, \Delta)$ admits a log-resolution $Y \to X$, then it is enough to require the above conditions only for one log-resolution $f : Y \to X$ \cite[2.32]{Kollar_Mori_Birational_geometry_of_algebraic_varieties}. In particular, this applies to dimensions at most three by \cite{Cutkosky_Resolution_of_singularities_for_3_folds_in_positive_characteristic,Cossart_Piltant_Resolution_of_singularities_of_threefolds_in_positive_characteristic_I,Cossart_Piltant_Resolution_of_singularities_of_threefolds_in_positive_characteristic_II}.
\end{remark}

\begin{exercise}
\label{ex:F_pure_implies_log_canonical_Cartier}
Show that for a sharply $F$-pure $X$ (that is, $(X, 0)$ is sharply $F$-pure) with $\omega_X$ a line bundle, $X$ is log canonical.
\Hint{Given a birational morphism $f : Y \to X$ of normal schemes, show first that there is commutative diagram
\begin{equation*}
\xymatrix{
f_* F_*^e (\omega_Y(p^e E)) =F_*^e f_* (\omega_Y (p^e E)) \ar[r] \ar[d] & f_* (\omega_Y(E)) \ar[d] \\
F_*^e \omega_X \ar[r] & \omega_X
}
\end{equation*}
for every integer $e >0$ and any exceptional $\bZ$-divisor $E$. Twist then this diagram with $\omega_X^{-1}$ and conclude the exercise.}
\end{exercise}

\begin{exercise}
\label{ex:changing_Delta}
Show that if $(X, \Delta)$ is globally $F$-split, then there is a $\Delta' \geq \Delta$, such that $(X, \Delta')$ is globally $F$-split, $(p^e-1)(K_X + \Delta')$ is Cartier. Furthermore, show that $\Delta'$ can be chosen so that $(X, \Delta')$ is log Calabi-Yau, that is,  $K_X + \Delta' \sim_{\bQ} 0$.
\Hint{Let $\Gamma$ be  the divisor given by the section of $\sO_X (\lfloor (1-p^e) (K_X + \Delta) \rfloor)$ for an $e>0$ for which there is a splitting. Using the dual picture (\autoref{exc:dual_form_of_global_F_split_F_regular}) show that this induces a splitting of $\sO_X \to F^e_* \sO_X(D + \lceil(p^e -1 ) \Delta \rceil )$. Define then $\Delta':= \frac{D + \lceil(p^e -1 ) \Delta \rceil }{p^e -1}$.}

(Note: There is a corresponding statement with globally $F$-regular and log-Fano, however the proof is more tedious. See \cite[Theorem 4.3]{Schwede_Smith_globally_F_regular_an_log_Fano_varieties}.)

\end{exercise}

\begin{exercise}
\label{exc:SFP_implies_lc}
Show that if $(X, \Delta)$ is a sharply $F$-pure pair such that $K_X + \Delta$ is $\bQ$-Cartier then $(X, \Delta)$ is log canonical.
\Hint{Soup up the argument of \autoref{ex:F_pure_implies_log_canonical_Cartier}. First using \autoref{ex:changing_Delta} assume that $(p^e -1)(K_X + \Delta)$ is $\bQ$-Cartier. Then construct a map
\begin{equation*}
\sF_e:=F^e_* \sO_Y( \lceil K_Y - f^* p^e (K_X+ \Delta) + \varepsilon p^e E \rceil)  \to \sF_0:= \sO_Y( \lceil K_Y - f^*  (K_X+ \Delta) + \varepsilon  E \rceil),
\end{equation*}
where $E$ is the reduced exceptional divisor. Then note, that
\begin{equation*}
f_* \sF_e \cong F^e_* \Big( \sO_X\big( (1-p^e)(K_X + \Delta) \big)   \otimes f_* \big(   \sO_Y( \lceil K_Y - f^* (K_X+ \Delta) + \varepsilon p^e E \rceil)  \big) \Big).
\end{equation*}
Show then that there is a commutative diagram as follows using the commutativity of the diagram on the regular locus.
\begin{equation*}
\left. \raisebox{22pt}{ \xymatrix{
f_* \sF^e \ar[r] \ar[d] & f_* \sF_0 \ar[d] \\
F_*^e \sO_X( (1-p^e)(K_X + \Delta) ) \ar[r] & \sO_X
}} \right) .
\end{equation*}}
\end{exercise}

\begin{exercise}
\label{exc:SFR_implies_klt}
Show that $(X, \Delta)$ is globally $F$-regular if and only if for every effective Cartier divisor $A$, $(X, \Delta + \varepsilon A)$ is globally $F$-split for some rational $\varepsilon>0$. Further, show that  $(X, \Delta)$ is klt if for every effective Cartier divisor $A$, $(X, \Delta + \varepsilon A)$ is log canonical for some $\bQ \ni \varepsilon>0$, and the other direction also holds if $(X, \Delta)$ admits a log-resolution (which is known up to dimension 3). Conclude that if $(X, \Delta)$ is a strongly $F$-regular pair with $K_X + \Delta$ being $\bQ$-Cartier, then $(X, \Delta)$ is klt.
\Hint{For the first statement again use the trick that showed up earlier that by replacing $e$ by $ne$ we may replace $D$ by $\frac{p^{ne}-1}{p^e-1} D$. For the second one given a birational morphism $f : Y \to X$,  choose an $A$ that contains the images of the exceptional divisors of $f$ on $X$. For the backwards direction of the second statement use \cite[Prop 2.36.(2)]{Kollar_Mori_Birational_geometry_of_algebraic_varieties}. For the conclusion use \autoref{exc:SFP_implies_lc}.}

\end{exercise}

Having shown \autoref{exc:SFR_implies_klt}, the question is how far are the two singularity classes. One answer to this is as follows.

\begin{theorem}
\label{thm:klt_SFR_reduction} \cite{Hara_Classification_of_two_dimensional_F_regular_and_F_pure_singularities,Mehta_Srinivas_A_characterization_of_rational_singularities,Hara_Geometric_interpretation_of_tight_closure_and_test_ideals,Smith_The_multiplier_ideal_is_a_universal_test_ideal}
If $(X, \Delta)$ is a klt pair over a field $k_1$ of characteristic zero, then for every model $(X_A, \Delta_A) \to \Spec A$ of $(X, \Delta) \to \Spec k_1$ over a finitely generated $\bZ$-algebra $A$, the set
\begin{equation*}
\{ p \in \Spec A| (X_p, \Delta_p) \textrm{ is strongly $F$-regular } \}
\end{equation*}
is open and dense.

\end{theorem}

Since the primary focus of these notes is the fixed characteristic situation, instead of going into the direction of \autoref{thm:klt_SFR_reduction}, we would like to understand the difference between klt and strongly $F$-regular in positive characteristic. This is mostly known only for surfaces \cite{Hara_Classification_of_two_dimensional_F_regular_and_F_pure_singularities}, which we will discuss thoroughly below. The higher dimensional case, as well as some surface questions on slc singularities are posed as research questions below.

So, take a klt surface singularity $X$. In any characteristic $X$ admits a minimal resolution $f: Y \to X$. The exceptional divisor of $f$ is a normal crossing curve, with all components being smooth rational curves and the dual graph of which is star shaped. A list of the possible dual graphs can be found on \cite[page 50-51]{Hara_Classification_of_two_dimensional_F_regular_and_F_pure_singularities}. It is presumably easier to work on $Y$ instead of $X$, since it is smooth with well-described normal crossing divisors. The next exercise is the first step in this direction.

\begin{exercise}
\label{exc:globally_F_regular_birational}
Let $f : (Y, \Gamma) \to (X, \Delta)$ be a proper birational morphism of pairs of dimension at least two, that is, we require $f_* \Gamma = \Delta$ (e.g., $Y$ is a resolution either of a variety or of a local ring of a variety $X$). Show the following statements.
\begin{enumerate}
\item Let $H \geq 0$ be a Cartier divisor on $X$ and $\Gamma \geq 0$ a divisor for which $f_* \Gamma=\Delta$. Assume further that
\begin{equation*}
H^0(Y, F^e_* \sO_Y(\lfloor (1-p^e)(K_Y+\Gamma) \rfloor - f^* H )) \to H^0(Y, \sO_Y)
\end{equation*}
is surjective for some integer $e > 0$. Then for the same value of $e$,
\begin{equation*}
H^0(X, F^e_* \sO_X(\lfloor (1-p^e)(K_X+\Delta) \rfloor - H )) \to H^0(X, \sO_X)
\end{equation*}
is surjective as well.

\item If $(Y, \Gamma)$ is globally $F$-split then so is $(X, \Delta)$.
\item If $(Y, \Gamma)$ is globally $F$-regular then so is $(X, \Delta)$.
\item Give a counterexample to the converse of the previous statement
\Hint{Fermat cubic surface.}
\end{enumerate}

\end{exercise}

\begin{proposition}
\label{prop:technical}
Consider the following situation:
\begin{enumerate}
\item \label{itm:technical:factorial} $(Y,D)$ a pair such that $Y$ is regular and $D$ is normal crossing (e.g., $Y$ is smooth over $k$),
\item \label{itm:technical:effective}  $G \geq 0$ a divisor on $Y$,
\item \label{itm:technical:center} $E$ is a prime divisor on $Y$, such that $\coeff_E D =1$,
\item \label{itm:technical:globallly_F_regular} $(E, (D -E)|_E)$ is globally $F$-regular,
\item \label{itm:technical:vanishing} for every integer $n \geq 0$,   $H^1(Y, \lceil (1-p^e)K_Y - p^e D \rceil  - G + n E) = 0$ for every $e \gg 0$.
\end{enumerate}
Then, the image of $H^0(Y, F^e_* \sO_Y( \lceil (1-p^e)K_Y - p^e D \rceil  - G + (n+1) E)) \to H^0(Y, \sO_Y)$ maps surjectively onto $H^0(E, \sO_E)$ for all $e \gg 0$, where $n$ is the biggest integer such that $nE \leq G$.
\end{proposition}

\begin{proof}
By \autoref{itm:technical:center},  $E$ is what we will call in Section  6.2 an $F$-pure center of $(X,D)$; all that is needed here is the commutativity of certain diagrams similar to those in Exercises 3.14 and 3.15.
Let $n$ be the biggest integers such that $nE \leq G$. Consider the following commutative diagram (note that $ \coeff_E \left( \lfloor p^e D \rfloor - (p^e -1)E \right) = 1$, hence the appearance of $n+1$ instead of $n$):
\begin{equation*}
\hspace{-30pt}
\xymatrix{
H^0(Y, F^e_* \sO_Y( \lceil (1-p^e)K_Y - p^e D \rceil  - G + (n+1) E)) \ar[d] \ar[r]^{\beta} & H^0(Y, F^e_* \sO_E( \lceil (1-p^e)K_E  -p^e (D-E)_E \rceil  - (G - n E)_E)) \ar[d] \ar@/^8pc/[dd]^{\alpha} \\
H^0(Y, F^e_* \sO_Y( (1-p^e)(K_Y+E) )) \ar[d] \ar[r] & H^0(Y, F^e_* \sO_E( (1-p^e)K_E )) \ar[d] \\
H^0(Y, \sO_Y) \ar[r] & H^0(E, \sO_E)}
\end{equation*}
We need to prove that both $\alpha$ and $\beta$ are surjective. The surjectivity of $\alpha$ follows from the global $F$-regularity assumption and the surjectivity of $\beta$ from the long exact sequence of cohomology and assumption \autoref{itm:technical:vanishing}.
\end{proof}

\begin{corollary}
\label{cor:globally_F_regular_with_lifting}
Let $Y$ be a smooth, projective variety over the algebraically closed ground field $k$. Assume further that $-(K_Y + D)$ is ample and $E$ is a codimension one $F$-pure center such that $(E, D-E)|_E)$ is globally $F$-regular. Then $(Y,D)$ and consequently $Y$ is globally $F$-regular.
\end{corollary}

\begin{exercise}
Show that $\bP^n$ is globally $F$-regular.
\Hint{use induction starting with $n=0$, apply \autoref{cor:globally_F_regular_with_lifting} in the inductional step.}
\end{exercise}

\begin{application}
\label{application:singularity}
We apply \autoref{prop:technical} to the situation when $X = \Spec A$ is the local ring of a normal singularity, $f : Y \to X$ is a log-resolution of singularities and $E \subseteq f^{-1}(P)$ is an adequately chosen prime divisor, where $P$ is the closed point of $X$. Assume also that $k$ is algebraically closed.  Since $X$ is local, it is strongly $F$-regular if and only if it is globally $F$-regular. By \autoref{exc:globally_F_regular_birational}, we should prove for every effective Cartier divisor $H$ on $X$ we are supposed to prove that the map
\begin{equation*}
H^0(Y, F^e_* \sO_Y( (1-p^e)K_Y  - f^*H )) \to H^0(Y, \sO_Y) \cong A
\end{equation*}
is surjective. For this it is enough to choose an adequate boundary $D$ on $Y$ (supported on the exceptional divisor), and show that
\begin{equation*}
H^0(Y, F^e_* \sO_Y( \lceil (1-p^e)K_Y - p^e D  \rceil  - f^*H )) \to H^0(Y, \sO_Y) \cong A.
\end{equation*}
At this point we would like to apply \autoref{prop:technical}. However, there we see that instead of $f^* H$ we would need to have $f^* H - nE$. The solution is to find such $D$ and $E$, such that $\lceil -p^eD \rceil + nE \leq 0$ and prove then that
\begin{equation*}
\psi : H^0(Y, F^e_* \sO_Y( \lceil (1-p^e)K_Y - p^e D ) \rceil  - f^*H + nE )) \to H^0(Y, \sO_Y) \cong A.
\end{equation*}
is surjective. Note that the condition $\lceil -p^eD \rceil + nE \leq 0$ is automatically satisfied for every $e \gg 0$, since we have chosen $\coeff_E D =1$. Further, note that the image of this map is an ideal $I$ in $A$. Since $A$ is local, then it is enough to prove that the image of $I$ in $A/m$ (where $m \subseteq A$ is the maximal ideal) is non-zero. That is, we want to prove that the image of $\psi$ maps surjectively onto $H^0(E, \sO_E)$. So, we just need to check that the conditions of \autoref{prop:technical} apply for adequately chosen $D$ and $E$.

Out of the conditions of \autoref{prop:technical}, \eqref{itm:technical:factorial} and \eqref{itm:technical:effective}  are automatically satisfied. For condition \eqref{itm:technical:center} we just need to choose $E$ to be one of the exceptional prime divisors in $D$ that has coefficient $1$. The other two conditions are harder, so we just summarize our findings:

$X$ is strongly $F$-regular, if for every effective Cartier divisor $H$ on $X$ we can find an exceptional effective divisor $D$, a component $E$ of $D$ with coefficient $1$  such that
\begin{enumerate}
\item $(E, (D -E)|_E)$ is globally $F$-regular,
\item for every integer $n \geq 0$,   $H^1(Y, \lceil (1-p^e)K_Y - p^e D) \rceil  - f^*H + n E) = 0$ for every $e \gg 0$.
\end{enumerate}
\end{application}

We show how \autoref{application:singularity} works for a two dimensional klt singularity. So, let $X$ be such a singularity and $f : Y \to X$ its minimal resolution. We want to use the following theorem.

\begin{theorem} \cite[2.2.1]{Kollar_Kovacs_Birational_geometry_of_log_surfaces}
\label{thm:vanishing}
If $f : Y \to X$ is a log resolution of a two dimensional normal  scheme essentially of finite type over $k$, and $L$ is an $f$-nef exceptional $\bQ$-divisor, then $H^1(Y, K_Y + \lceil L \rceil )= 0$.
\end{theorem}

First we work through one concrete example, the canonical $D_5$ surface singularity, and leave the general case for exercises. Consider a canonical $D_5$ surface singularity. So, the dual graph is
\begin{equation}
\label{eq:D_5}
\xymatrix{
E_1  \ar@{-}[dr] \\
& E(=E_0) \ar@{-}[r] & E_3 \ar@{-}[r] & E_4 \\
E_2 \ar@{-}[ur] \\
}
\end{equation}
where all the self-intersections are $-2$ and all the exceptional  curves are smooth rational curves. Let $E$ be the prime divisor corresponding to the center of this graph as shown on \autoref{eq:D_5}. According to \autoref{thm:vanishing} we want
\begin{equation*}
(1-p^e)(K_Y+ D)   - f^* H + n E
\end{equation*}
to be nef for $n\geq 0$ and every $ e \gg 0$ (in particular, $K_Y + D$ to be anti-nef). Further we need $\coeff_E D=1$ and $(E, (D-E)|_E)$ to be globally $F$-regular. It turns out that there is a good choice:
\begin{equation*}
D:= \frac{1}{2}E_1 + \frac{1}{2} E_2 + \frac{2}{3} E_3 + \frac{1}{3} E_4 + E
\end{equation*}
which can be pictured as
\begin{equation}
\label{eq:D_5}
\xymatrix{
\frac{1}{2}  \ar@{-}[dr] \\
& 1 \ar@{-}[r] & \frac{2}{3} \ar@{-}[r] & \frac{1}{3} \\
\frac{1}{2} \ar@{-}[ur] \\
}
\end{equation}
Indeed, the above choice of $D$ and $E$ is adequate, because:
\begin{enumerate}
\item $(E, (D-E)_E) \cong \left(\bP^1, \frac{1}{2}P_1 +  \frac{1}{2}P_2  + \frac{2}{3}P_3 \right)$, which is globally $F$-regular by \autoref{exc:explicit_gfr} if $p \neq 2$.
\item By the adjunction formula $(K+ D) \cdot E_i = -2 - E_i^2 + D \cdot E_i = D \cdot E_i$, therefore $(K + D) \cdot E_i = -2 \coeff_{E_i} D + \sum_{E_i \cap E_j \neq \emptyset} \coeff_{E_j} D$ and hence $(K +D) \cdot E_i = 0$ for $i \neq 0$ and $(K +D) \cdot E = - \frac{1}{3}$.
In particular, $-p^e( K_Y  +  D)  + n E$ is $f$-nef for every $n \geq 0$ and  $e \gg 0$. Further,
\begin{equation*}
 \lceil (1-p^e)K_Y - p^e D) \rceil   - f^*H + n E =  K_Y +  \lceil-p^e( K_Y  +  D)    - f^*H + n E \rceil ,
\end{equation*}
hence by \autoref{thm:vanishing}, $H^1(Y, \lceil (1-p^e)(K_Y+ D) \rceil   - G + n E) = 0$  for every $n \geq 0$ and $e \gg 0$.

\end{enumerate}

So, we have just proven that:

\begin{theorem} \cite{Hara_Classification_of_two_dimensional_F_regular_and_F_pure_singularities}
A klt surface singularity of type $D_5$ is strongly $F$-regular if $p \neq 2$.
\end{theorem}

The general case (the singularities other than canonical $D_5$) is left for the following few exercises. So, let $X= \Spec A$ be a general klt surface singularity, where $A$ is a local ring.  As mentioned above the dual graph of the exceptional curve is star shaped. The star can have either two or three arms. In the three arm case we group the graphs according to the absolute value of the determinants of the intersection matrices of the arms. Then the possible cases are $(2,2,d)$, $(2,3,3)$, $(2,3,4)$ and $(2,3,5)$. We set $E=E_0$ to be the vertex of the star in general, and we define $D$ by the properties
\begin{equation*}
\coeff_E D =1 \ ; \  (K_Y + D) \cdot E_i = 0 \  (i \neq 0).
\end{equation*}
Let $E_1$, $E_2$ (and possibly $E_3$) be the curves in the arms that are adjacent to $E=E_0$.

\begin{exercise}
Show that
\begin{equation*}
D= E_0 + \sum_i \frac{d_i-1}{d_i} E_i + \dots
\end{equation*}
where $(d_1,d_2,d_3)$ is the type of the graph.
Conclude then that $(K+D) \cdot E = \sum_i \frac{d_i-1}{d_i} -1 < 0$.
\vskip 3pt
\Hint{Consider the intersection matrix of the $i$-th arm:
\begin{equation*}
\begin{pmatrix}
e_1 & 1 & 0 & \dots & & \\
1 & e_2 & 1 & 0 & \dots & \\
0 &1 & e_3 & 1 & 0 & \dots  \\
\vdots & \vdots & \vdots & \\
\end{pmatrix}
\end{equation*}
Using the adjunction formula, write up a linear equation for computing the coefficients of $D$ in the $i$-th arm involving the above matrix. Then use Cramer's rule and that the above matrix has determinant $d_i$.}
\end{exercise}

\begin{exercise}
Show that the followings are globally $F$-regular.
\begin{enumerate}
\item $\left( \bP^1 , \frac{1}{2} P_1 + \frac{1}{2} P_2  \right)$
 \item $\left( \bP^1 , \frac{1}{2} P_1 + \frac{2}{3} P_2 + \frac{2}{3} P_3 \right)$for $p \neq 2,3$
 \item $\left( \bP^1 , \frac{1}{2} P_1 + \frac{2}{3} P_2 + \frac{3}{4} P_3 \right)$for $p \neq 2,3$
 \item $\left( \bP^1 , \frac{1}{2} P_1 + \frac{2}{3} P_2 + \frac{4}{5} P_3 \right)$for $p \neq 2,3,5$
\end{enumerate}

\Hint{Use \autoref{exc:explicit_gfr_general}.}
\end{exercise}

So, we have proved the theorem:

\begin{theorem}
  \cite{Hara_Classification_of_two_dimensional_F_regular_and_F_pure_singularities}
Every klt surface singularity  is strongly $F$-regular if $p > 5$.
\end{theorem}

\subsection{Questions}

\begin{question} {\scshape Depth of sharply $F$-pure singularities.}
Let $(X, \Delta)$ be a log canonical pair (allowing $\Delta$ to be an $\bR$-divisor) and let $x \in X$ be a point which is not a log canonical center of $(X, \Delta)$. Suppose that $0 \leq \Delta' \leq \Delta$ is another $\bQ$-divisor and $D$ a $\bZ$-divisor, such that $D \sim_{\bQ, \mathrm{loc}} \Delta'$. Show then that
\[
\depth_x \sO_X(-D) \geq \min\{3, \codim_X x\}.
\]
or give a counterexample. Note that in characteristic zero this was shown in \cite{Kollar_A_local_version_of_the_Kawamata_Viehweg} and in positive characteristic $p>0$ in \cite{Patakfalvi_Schwede_Depth_of_F_singularities} with the extra assumption that the indices are not divisible by $p$ and the singularities are sharply $F$-pure. So, the interesting case is when either $p$ divides some of the appearing indices or the singularities are log canonical but not sharply $F$-pure. In two or three dimensions this could be in particular approachable by an argument similar to that of \cite{Kollar_A_local_version_of_the_Kawamata_Viehweg}, since then some of the required vanishings can still hold (because the dimension is small). Further, the above question is also interesting for slc singularities (see \autoref{subsubsec:slc}) as stated in \cite[Thm 3]{Kollar_A_local_version_of_the_Kawamata_Viehweg}.

Other related question is that in the above situation if $Z$ is the reduced subscheme supported on the union of some log canonical centers, then show or disprove
\begin{equation*}
\depth_x \sI_Z \geq \min\{3, \codim_X x + 1 \}.
\end{equation*}
The known and unknown cases for this question are completely analogous to the previous one (see \cite[Thm 3]{Kollar_A_local_version_of_the_Kawamata_Viehweg} and \cite[Thm 3.6]{Patakfalvi_Schwede_Depth_of_F_singularities}).

A third  related question is to show or disprove that if $(X, \Delta)$ has dlt singularities and $0 \leq D$ a $\bZ$-divisor such that $\Delta \sim_{\bQ, \mathrm{loc}} D$, then  $\sO_X(-D)$ is Cohen-Macaulay. The situation here is also analogous (see \cite[Thm 2]{Kollar_A_local_version_of_the_Kawamata_Viehweg} and \cite[Thm 3.1]{Patakfalvi_Schwede_Depth_of_F_singularities}) .

\end{question}

\begin{question}
\label{qtn:Frobenius_stable_Grauert_Riemenschneider}
{\scshape Frobenius stable Grauert-Riemenschneider vanishing.} If $f : Y \to X$ is a resolution of a normal variety, then is the natural Frobenius map $F^e_* R^i f_* \omega_Y  \to R^i f_* \omega_Y$ zero for any $i>0$ and $e \gg 0$?
\end{question}

\begin{exercise}
Check \autoref{qtn:Frobenius_stable_Grauert_Riemenschneider} for $X$ being a cone over a smooth projective variety.
\end{exercise}

\subsubsection{Three dimensional singularities}

The main question, which is probably hard in general, but should be kept in mind as a guiding principle, is the following:

\begin{question}
\label{qtn:klt_equalst_SFR_in_high_diemnsions}
In a fixed dimension $n$ is there a prime $p_0$, such that for all $p \geq p_0$, klt is equivalent to strongly $F$-regular for pairs with zero boundaries? What if we allow boundaries with coefficients in a finite set?
\end{question}

Most likely, there is either an easy counterexample to \autoref{qtn:klt_equalst_SFR_in_high_diemnsions} or it is extremely hard. So, first we should set our sights to dimension 3:

\begin{question}
Answer \autoref{qtn:klt_equalst_SFR_in_high_diemnsions} for $n=3$. One can also restrict further to terminal singularities. That is, is there a prime $p_0$, such that for all $p \geq p_0$, every terminal singularity $X$ is strongly $F$-regular? The latter question seems to be more approachable, since in characteristic zero there is a good classification of three dimensional terminal singularities \cite{Mori_On_3_dimensional_terminal_singularities}. So, first the question would be how much of that holds in positive characteristic. However, even without solving this question, it would be interesting to take examples of log-resolutions in characteristic zero that show up in any characteristic and find the primes for which they are strongly $F$-regular. One can use that in characteristic zero we know nice resolutions of terminal singularities (c.f., \cite{Chen_Explicit_resolution_of_three_dimensional_terminal_singularities}).
\end{question}

\subsubsection{Classify slc surface singularities}
\label{subsubsec:slc}

Slc (semi-log canonical) singularities are the natural non-normal versions of log canonical singularities. Their main use, as well as the motivation for their introduction \cite{Kollar_Shepher_Barron_Threefolds_and_deformations}, is in higher dimensional moduli theory. As in dimension one needs to add nodal curves to smooth curves to obtain a projective moduli space of curves, in higher dimensions one has to add slc singularities to log canonical singularities to obtain a (partially proven and partially conjectural) moduli space of canonical models. Note that such moduli space parameterizes birational equivalence classes of varieties of general type (which is shown in characteristic zero, for surfaces in positive characteristic and for threefolds in characteristic $p \geq 7$).

First, we define slc singularities. Let $k$ be a field.

\begin{definition}
A scheme $X$ locally of finite type over $k$ is \emph{demi-normal} if it is equidimensional, $S_2$ and its codimension one points are nodes \cite[5.1]{Kollar_Singularities_of_the_minimal_model_program}. Here $x \in X$ is a \emph{node}, if $\sO_{X,x}$ is isomorphic to $R/(f)$ for some two dimensional regular local ring $(R,m)$, $f \in m^2$ and $f$ is not a square in $m^2/m^3$ \cite[1.41]{Kollar_Singularities_of_the_minimal_model_program}.
\end{definition}

\begin{proposition} \cite[1.41.1]{Kollar_Singularities_of_the_minimal_model_program}
Let $(A,n)$ be a 1-dimensional local ring with residue field $k$ and normalization $\overline{A}$. Let $\overline{n}$ be the intersection of the maximal ideals of $\overline{A}$. Assume that $(A, n)$ is a quotient of a regular local ring. Then $(A,n)$ is nodal if and only if $\dim_k (\overline{A}/\overline{n})=2$ and $\overline{n} \subseteq A$.
\end{proposition}

\begin{exercise}
Show that the pinch point $\Spec \frac{ k[x,y,t]}{(x^2 -t y^2)}$ is demi-normal, even in characteristic 2.
\end{exercise}

\begin{definition} \cite[5.2]{Kollar_Singularities_of_the_minimal_model_program}
If $X$ is a reduced scheme, and $\pi : \oX \to X$ its normalization, then the conductor ideal is the largest ideal sheaf $\sI \subseteq \sO_X$ that is also an ideal sheaf of $\pi_* \sO_{\oX}$. The conductor subschemes are defined as $D:=\Spec ( \sO_X/\sI)$ and $\oD:=\Spec (\sO_{\oX}/\sI)$.
\end{definition}

\begin{fact} \cite[5.2]{Kollar_Singularities_of_the_minimal_model_program}
For a demi-normal scheme $X$, the conductor subschemes are the reduced divisors supported on  the closure of the nodal locus, and on the preimage of that, respectively.
\end{fact}

\begin{definition} \cite[5.10]{Kollar_Singularities_of_the_minimal_model_program}
Let $X$ be a   demi-normal scheme with normalization $\pi : \oX \to X$ with conductor divisors $D \subseteq X$ and $\oD \subseteq \oX$, respectively. Let $\Delta$ be a $\bQ$ (Weil-)divisor on $X$ not containing common components with $D$, and let $\overline{\Delta}:= \pi^* \Delta$ (defined as the divisorial part of $\pi^{-1}(\Delta)$). Then the pair $(X, \Delta)$ is \emph{semi-log canonical}, or shortly just \emph{slc}, if $K_X + \Delta$ is $\bQ$-Cartier, and $(\oX, \oD + \overline{\Delta})$ is log canonical.
\end{definition}

Now, we would like to see the $F$-singularities aspects of slc singularities. Though we defined $F$-purity in the normal setting, it works without modification in the $G_1$ and $S_2$ setting (c.f., \cite{Patakfalvi_Schwede_Depth_of_F_singularities}).

\begin{exercise}
Show that if $p=2$, then the 
pinch-point ($\{x_1^2 + x_2^2x_3=0 \} \subseteq \bA^n_{x_1,\dots,x_n}$, $n \geq 3$) is not $F$-pure along the entire double line (i.e., along $x_1=x_2=0$).
\vskip 3pt
\Hint{Use \autoref{prop:Fedder_s_criterion}.}
\end{exercise}

A more general wording of the previous exercise is as follows.

\begin{exercise}
Let  $X$ be an affine variety (where we do not assume that $X$ is irreducible) and let $x$ be a codimension one point of the conductor given by the conductor ideal $\sI \subseteq \sO_X$.
\begin{enumerate}
\item \label{itm:conductor} Show that for every $\phi \in \Hom_{\sO_X}( F^e_* \sO_X , \sO_X)$, $\phi(F^e_* \sI)\subseteq \sI$.
\item \label{itm:extension} Show that  every $\phi \in \Hom_{\sO_X}( F^e_* \sO_X , \sO_X)$ extends to $\phi^N \in \Hom_{\sO_Y}( F^e_* \sO_Y , \sO_Y)$, where $Y$ is the normalization of $X$.
\item \label{itm:conductor_inseparable} Show that the conductor of the normalization is inseparable at $x$ over the conductor of $X$, then $X$ is not $F$-pure at $x$.
\end{enumerate}

\Hint{For point \autoref{itm:conductor} extend $\phi$ first to the total ring of fractions $K$ of $R:= \Gamma(X,\sO_X)$, then using this extension show  that $\phi(s) t \in I:=\Gamma(X,\sI)$ for every $s \in I$ and $t \in R^N:=\Gamma(Y, \sO_Y)$.

For point \autoref{itm:extension} show that if $\phi'$ is the above extension  of $\phi$ to $K$, then $ \phi'(F^e_* R^N) \subseteq R^n$. For that first show that we may assume that $R$ is integral. Then show that $(\phi'(x))^m I  \subseteq I$ for every $x \in F^e_* R^N$ and integer $m>0$. Then deduce that $\phi'(x)$ is integral over $R$.

For point \autoref{itm:conductor_inseparable} assume that $X$ is $F$-pure.  Take then a splitting $\phi$ of $\sO_X \to F_* \sO_X$ guaranteed by the $F$-purity. Let $\phi^N$ be the splitting given by point \autoref{itm:extension}. Let $B$ be the conductor on $X$ and $C$ the conductor on $Y$. Show then using point \autoref{itm:conductor} that this yields a splitting of $\sO_B \to F_* \sO_B$ that extends to a splitting of $\sO_C \to F_* \sO_C$. The latter yields then a contradiction with the inseparability. }
\end{exercise}

The following are the questions posed about slc singularities.

\begin{question}
Classify slc surface singularities in characteristic $p>0$. There are more approaches to the  characteristic zero classification. Historically, one first passes to the canonical covers and then classifies only the Gorenstein slc surface singularities \cite[4.21, 4.22, 4.23, 4.24]{Kollar_Shepher_Barron_Threefolds_and_deformations}.  The second goes by classifying the normalization (c.f., \cite[Section 3.3]{Kollar_Singularities_of_the_minimal_model_program}) and then using Koll\'ar's gluing theory \cite[5.12, 5.13]{Kollar_Singularities_of_the_minimal_model_program} to see when one can glue  the normalization into a demi-normal singularity. Unfortunately, there are issues with both  in positive characteristic. For the first one, the canonical cover can be inseparable in characteristic $p$, and for the second one, aspects of the gluing theory work only in characteristic zero \cite[5.12, 5.13]{Kollar_Singularities_of_the_minimal_model_program}.
\end{question}

\begin{question}
Find the $F$-pure ones out of the list of slc singularities found as an answer to the previous question. For the method see \cite{Hara_Classification_of_two_dimensional_F_regular_and_F_pure_singularities,Mehta_Srinivas_Normal_F_pure_surface_singularities} and \cite{Miller_Schwede_Semi_log_canonical}. Note that the two articles together solve the question if the conductor of the normalization maps separably to the conductor of the singularity (which is always true if $p \neq 2$) and further the index is not divisible by $p$. So, the interesting cases are the ones left out.
\end{question}

%

\section{Global Applications}
\label{sec:global}

In this section we present one application to global geometry of the theory of $F$-singularities and then we list a few problems. The known applications are somewhat different in nature with little connection between some of them (a non-complete list is: \cite{Schwede_A_canonical_linear_system,Hacon_Singularities_of_pluri_theta_divisors_in_Char_p,Mustata_The_non_nef_locus_in_positive_characteristic,Mustata_Schwede_A_Frobenius_variant_of_Seshadri_constants,Cascini_Hacon_Mustata_Schwede_On_the_numerical_dimension_of_pseudo_effective_divisors_in_positive_characteristic,Hacon_Xu_On_the_three_dimensional_minimal_model_program_in_positive_characteristic,CasciniTanakaXuOnBPF,ZhangPluriCanonicalMapOfVarietiesOfMaximalAlbaneseDimension,Patakfalvi_Semi_positivity_in_positive_characteristics,Patakfalvi_On_subadditivity_of_Kodaira_dimension_in_positive_characteristic,Hacon_Patakfalvi_A_generic_vanishing_in_positive_characteristic}). In particular, the chosen application is admittedly somewhat random.

Throughout the section the base field $k$ is algebraically closed, and of characteristic $p>0$. The algebraically closed assumption is not necessary in \autoref{subsec:semi_positivity} but it is important in \autoref{subsec:log_Fanos} among other places.

\subsection{Semi-positivity of pushforwards}
\label{subsec:semi_positivity}

The application we discuss in more details here is the semi-positivity of sheaves of the form $f_* \left(\omega_{X/Y}^m \right)$. We use the following notations. This is very far from the most general situation where the results work (see \cite{Patakfalvi_Semi_positivity_in_positive_characteristics} for a more general situation).

\begin{notation}
\label{notation:semi_positivity}
 $X$ is a Gorenstein, projective variety, $f : X \to Y$ is a surjective, morphism  to a smooth projective curve with normal  (and connected) geometric generic fiber, and hence the fibers are normal (and connected) over an open set of $Y$. Fix also a closed point $y_0 \in Y$ such that  $X_0:= X_{y_0}$ is normal.
\end{notation}

We prove the following theorem. Recall that a vector bundle $\sE$ on a projective scheme $Z$ is nef, if for all finite maps $\tau : C \to Z$ from smooth curves, and all line bundle quotients $\tau^* \sE \twoheadrightarrow \sL$, $\deg \sL \geq 0$.

\begin{theorem} \cite{Patakfalvi_Semi_positivity_in_positive_characteristics}
\label{thm:semi_positivity}
In the situation of \autoref{notation:semi_positivity}, if $\omega_{X/Y}$ is $f$-ample and $X_0$ is sharply $F$-pure, then $f_* (\omega_{X/Y}^m)$ is a nef vector bundle for every $m \gg 0$.

\end{theorem}

Note that the corresponding theorem in characteristic zero was proven using the consequence of  Hodge theory that $f_* \omega_{X/Y}$ is semi-positive. The corresponding statement in characteristic $p$ is false, so the proof necessarily eludes such considerations.

Recall that a coherent sheaf $\sF$ on a scheme $X$ is \emph{generically globally generated}, if there is homomorphism $\sO_X^{\oplus m} \to \sF$ which is surjective over a dense open set. Further, if $Z$ is a normal variety and $\sL$ a line bundle on $Z$, then  we use the notation $S^0(Z, \sL)$ for $S^0(Z, \sigma(X, 0) \otimes \sL)$.

\begin{proposition}
\label{prop:generically_globally_generated}
In the situation of  \autoref{notation:semi_positivity}, choose a Cartier divisor $N$ and set $\sN:=\sO_X(N)$. Assume that $N - K_{X/Y} $ is $f$-ample and nef, and  $H^0(X_0,\sN|_{X_0})= S^0(X_0, \sN|_{X_0})$.
Then $f_* \sN \otimes \omega_Y(2y_0)$ is generically globally generated.
\end{proposition}

\begin{proof}
Set $M:=N + f^* K_Y + 2 X_0$ and $\sM:=\sO_X(M)$. Consider the commutative diagram below.
\begin{equation}
\label{eq:generically_globally_generated:commutative_diagram}
\xymatrix{
f_* \sM \ar[r] & (f_* \sM) \otimes k(y_0) \ar@{^(->}[r] &  H^0(X_0, \sM|_{X_0})  \\
S^0( X, \sigma(X,  X_0) \otimes \sM)  \otimes \sO_Y \ar[u] \ar[rr] & & S^0(X_0, \sM|_{X_0}) \ar@{=}[u]
} ,
\end{equation}
where $H^0(X_0, \sM|_{X_0}) = S^0(X_0,  \sM|_{X_0})$, because
\begin{equation*}
 H^0(X_0, \sM|_{X_0})  \cong  H^0(X_0, \sN|_{X_0}) = S^0(X_0,  \sN|_{X_0}) \cong S^0(X_0,  \sM|_{X_0}).
\end{equation*}
Note that
\begin{equation}
\label{eq:generically_globally_generated:divisors}
M - K_X  - X_0
=
N + f^* K_Y + 2 X_0 - K_{X/Y} - f^* K_Y  - X_0
=
N - K_{X/Y}  + X_0 .
\end{equation}
Note also that $N - K_{X/Y} $ is nef and $f$-ample by assumption. Furthermore, $X_0$ is the pullback of an ample divisor from $Y$. Hence, $N - K_{X/Y} + X_0$  is ample and then by \autoref{eq:generically_globally_generated:divisors} so is $M - K_X  - X_0$.
Hence, \autoref{eq.LiftingDiagram}  implies that the bottom horizontal arrow in \autoref{eq:generically_globally_generated:commutative_diagram} is surjective. This finishes our proof.
\end{proof}

\begin{exercise}
\label{exc:generic_global_generation_nef}
If $\sF$ is a vector bundle on a smooth curve $Y$ and $\sL$ is a line bundle such that for every $m>0$, $\left( \bigotimes_{i=1}^m\sF \right) \otimes \sL$ is generically globally generated, then $\sF$ is nef.
\end{exercise}

\begin{notation}
\label{notation:product}
For a morphism $f : X \to Y$ of schemes, define
\begin{equation*}
X^{(m)}_Y :=\underbrace{X \times_Y X \times_Y \dots \times_Y X}_{\textrm{$m$
times}} ,
\end{equation*}
and let $f^{(m)}_Y : X^{(m)}_Y \to Y$ be the natural induced map. If $\sF$ is a sheaf of $\sO_X$-modules, then
\begin{equation*}
\sF^{(m)}_Y := \bigotimes_{i=1}^m p_i^* \sF,
\end{equation*}
where $p_i$ is the $i$-th projection $X^{(m)}_Y \to X$.
In most cases, we omit $Y$ from our notation. I.e., we use $X^{(m)}$, $f^{(m)}$ and $\sF^{(m)}$ instead of $X^{(m)}_Y$,  $f^{(m)}_Y$ and $\sF^{(m)}_Y$, respectively.
\end{notation}

\begin{exercise}
\label{exc:product}
If $\sN$ is a line bundle on a normal Gorenstein variety $X$, then
\begin{equation*}
S^0 \left( X^{(m)},  \sN^{(m)} \right) \cong S^0(X, \sN)^{\otimes m}.
\end{equation*}
(Here $X^{(m)}$ and $\sN^{(m)}$ are taken over $\Spec k$.)
\vskip 3pt
\Hint{the main issue is showing that the trace map is the box product of the trace maps, it is easier to show it on the regular locus and then just extend globally since the complement has large codimension.}
\end{exercise}

\begin{exercise}
\label{exc:pushforward}
Show that in the situation of \autoref{notation:product}, $f_*^{(m)}  \sN^{(m)} \cong \bigotimes_{i=1}^m f_* \sN$ for any line bundle $\sN$ on $X$ and integer $m >0$
\Hint{do induction, for one induction step use the projection formula and flat base-change.}
\end{exercise}

\begin{proposition}
\label{prop:semi_positive}
In the situation of Notation \ref{notation:semi_positivity}, choose a Cartier divisor $N$ and set $\sN:=\sO_X(N)$.
Assume that $N - K_{X/Y} $ is nef and $f$-ample, and $H^0(X_0,\sN|_{X_0})= S^0(X_0,  \sN|_{X_0})$.
Then $f_* \sN $ is a nef vector bundle.
\end{proposition}

\begin{proof}
First we claim that $X^{(n)}$ is a variety, that is, it is integral for every integer $n>0$. Indeed, since $X$ is Gorenstein and $Y$ is smooth, $X$ is relatively Gorenstein over $Y$, and hence so is $X^{(n)}$. Therefore, $X^{(n)}$ is also absolutely (so not only relatively over $Y$) Gorenstein. In particular, $X$ is $S_2$ and then we can check reducedness only at generic points. However, $X^{(n)}$ is flat over $Y$, so all generic points lie over the generic point of $Y$. Further over the generic point of $Y$,  $X^{n}$ is reduced and irreducible since the geometric generic fiber of $f$ is assumed to be normal (and connected). This concludes our claim.

Hence the assumptions of \autoref{notation:semi_positivity} are satisfied for $f^{(n)} : X^{(n)} \to Y$, $N^{(n)} - K_{X^{(n)}/Y} = (N - K_{X/Y})^{(n)}$ is nef and $f^{(n)}$-ample, and further by \autoref{exc:product} and the K\"unneth formula,
\begin{equation*}
\qquad H^0 \left( X_0^{(n)},\sN^{(n)}|_{X_0^{(n)}} \right)
=
H^0\left(X_0,\sN|_{X_0} \right)^{\otimes n}  \cong   S^0\left(X_0,  \sN|_{X_0}\right)^{\otimes n}
\cong
S^0 \left( X_0^{(n)}, \sN^{(n)}|_{X_0^{(n)}} \right).
\end{equation*}
Hence Proposition \ref{prop:generically_globally_generated} applies to $X^{(n)}$ and $N^{(n)}$, and consequently, $f^{(n)}_*( \sN^{(n)} ) \otimes \omega_Y(2y_0)$ generically globally generated for every $n>0$.

By \autoref{exc:pushforward}, $f^{(n)}_* \left( \sN^{(n)} \right) \cong \bigotimes_{i=1}^n f_* \sN$. Therefore, $f_* \sN$ is a vector bundle, such that $\left( \bigotimes_{i=1}^n f_* \sN \right) \otimes \omega_Y(2y_0)$ is generically globally generated for every $n>0$. Hence, by  \autoref{exc:generic_global_generation_nef}, $f_* \sN$ is a nef vector bundle. This concludes our proof.

\end{proof}

\begin{proposition}
\label{prop:nef}
In the situation of Proposition \ref{prop:semi_positive}, if furthermore
  $\sN_y$ globally generated for all $y \in Y$,
 then $\sN$ is nef.

\end{proposition}

\begin{proof}
Consider the following  commutative diagram for every $y \in Y$.
\begin{equation}
\label{eq:nef:diagram}
\xymatrix{
f^* f_* \sN \ar[d] \ar[r] &  \sN  \ar[d] \\
H^0(X_y, \sN) \otimes \sO_{X_y} \ar[r] & \sN_y
}
\end{equation}
The left vertical arrow is an isomorphism for all but finitely many $y$ by cohomology and base-change. The bottom horizontal arrow is surjective for all $y \in Y$ by assumption. Hence $f^* f_* \sN \to \sN$ is surjective except possibly at points lying  over finitely many points of $y \in Y$. To show that $\sN$ is nef, we have to show that $\deg ( \sN|_C )\geq 0$ for every smooth projective curve $C$ mapping finitely to $X$. By assumption this follows if $C$ is vertical (globally generated implies nef). So, we may assume that $C$ maps surjectively onto $Y$. However, then $(f^* f_* \sN)|_C \to \sN|_C$ is generically surjective. Since $f_* \sN$ is nef by Proposition \ref{prop:semi_positive}, so is $f^* f_* \sN$ and hence $\deg (\sN|_C) \geq 0$ has to hold.
\end{proof}

The next exercises tell us how to satisfy the condition $H^0=S^0$ from the previous proposition.

\begin{exercise}
Let $(X,\Delta)$ be a pair such that $(p^e -1)(K_X+\Delta)$ is Cartier for some $e>0$. Show then that the restriction of the natural map $F^{ne}_* \sO_X((1-p^{ne})(K_X + \Delta)) \to \sO_X$ to $\sigma(X, \Delta) \otimes \sO_X((1-p^{ne})(K_X + \Delta)) $ yields for every $n \gg 0$ a surjective homomorphism
\begin{equation*}
 F^{ne}_* (\sigma(X, \Delta) \otimes \sO_X((1-p^{ne})(K_X + \Delta)) ) \to \sigma(X, \Delta).
\end{equation*}
\Hint{Since the image of the original homomorphism is $\sigma(X, \Delta)$ for every $ n \gg0$, it is clear that the image of the restricted homomorphism is contained in $\sigma(X, \Delta)$. To prove surjectivity, take an $n'$, such that
\begin{equation*}
 F^{n'e}_* ( \sO_X((1-p^{n'e})(K_X + \Delta)) ) \to \sigma(X, \Delta).
\end{equation*}
is surjective and show that the composition of
\begin{multline*}
 F^{ne}_* (  F^{n'e}_* ( \sO_X((1-p^{n'e})(K_X + \Delta)) )  \otimes \sO_X((1-p^{ne})(K_X + \Delta)) ) \to
\\
\to F^{ne}_* ( \sigma(X, \Delta) \otimes \sO_X((1-p^{ne})(K_X + \Delta)) )
 \to \sigma(X, \Delta)
\end{multline*}
is also surjective.}
\end{exercise}

\begin{exercise}
\label{exc:alternate_def_of_S_0}
 Show that if $L$ is a  line bundle on a projective pair $(X,\Delta)$ such that $(p^e -1)(K_X+\Delta)$ is Cartier for some $e>0$, then for each $n \gg 0$, $S^0(X, \sigma (X, \Delta) \otimes L)$ is equal to the image of the homomorphism
\begin{equation*}
 H^0(X, L \otimes F^{ne}_* (\sigma(X, \Delta) \otimes \sO_X((1-p^{ne})(K_X + \Delta)) ) \to H^0(X, L \otimes \sigma(X, \Delta)).
\end{equation*}
\Hint{Let $V_n$ be the image of the above homomorphism. It is immediate that $V_n \subseteq S^0(X, \sigma(X, \Delta) \otimes L)$. For the other containment, use a trick similar to the previous exercise: take an integer $n'>0$, such that
\begin{equation*}
 F^{n'e}_* ( \sO_X((1-p^{n'e})(K_X + \Delta)) ) \to \sigma(X, \Delta).
\end{equation*}
is surjective, and consider the composition of
\begin{multline*}
 H^0(X, L \otimes F^{ne}_* (  F^{n'e}_* ( \sO_X((1-p^{n'e})(K_X + \Delta)) )  \otimes \sO_X((1-p^{ne})(K_X + \Delta)) )) \to \\
\to H^0(X, L \otimes F^{ne}_* ( \sigma(X, \Delta)  \otimes \sO_X((1-p^{ne})(K_X + \Delta)) )) \to H^0(X, L \otimes \sigma(X, \Delta)).
\end{multline*}
Show that the image of the above homomorphism contains $V_n$, and also that this image equals $S^0(X, \sigma(X, \Delta) \otimes L)$ for $n$ big enough.}
\end{exercise}

\begin{exercise}
\label{exc:S_0_equals_H_0}
Show that if $L$ is an ample line bundle on a projective pair $(X,\Delta)$ such that $(p^e -1)(K_X+\Delta)$ is Cartier for some $e>0$, then there is an integer $n>0$, such that for every nef line bundle $N$,  $S^0(X, \sigma(X, \Delta) \otimes L^n \otimes N)= H^0(X, \sigma(X, \Delta) \otimes L^n \otimes N)$.

\Hint{Let $e>0$ be an integer such that $(p^e-1)(K_X + \Delta)$ is Cartier. According to \autoref{exc:alternate_def_of_S_0}, it is enough to show that
\begin{multline*}
H^0 \left(X, L^n \otimes N \otimes F^{(i+1)e}_* \left(\sigma(X, \Delta) \otimes \sO_X((1-p^{(i+1)e})(K_X + \Delta))\right) \right)
\\ \to H^0 \left(X, L^n \otimes N \otimes F^{ie}_* \left(\sigma(X,\Delta) \otimes \sO_X((1-p^{ie})(K_X + \Delta))\right)\right)
\end{multline*}
is surjective for every $i$. Show then that the above maps are induced from the exact sequence
\begin{equation*}
\xymatrix{
0 \ar[r] & \sB \ar[r] & F^e_* ( \sigma(X,\Delta) \otimes \sO_X ( ( 1-p^e) (K_X + \Delta))) \ar[r] & \sigma(X, \Delta) \ar[r] & 0.
}
\end{equation*}
Use then Fujita vanishing \cite{FujitaVanishingTheoremsForSemiPositive} to show surjectivity.}
\end{exercise}

\begin{theorem}
\label{thm:relative_canonical_nef2}
In the situation of Notation \ref{notation:semi_positivity}, if  $X_0$ is sharply $F$-pure  and $K_{X/Y}$ is $f$-nef, then $K_{X/Y}$ is nef.
\end{theorem}

\begin{proof}
Using relative Fujita vanishing \cite{Keeler_Fujita_s_conjecture_and_Frobenius_amplitude} and \autoref{exc:S_0_equals_H_0}, there is an ample enough line bundle $\sQ$ on $X$, such that for all $i>0$  and $f$-nef line bundle $\sK$,
\begin{equation}
\label{eq:relative_canonical_nef2:assumption_central_fiber}
  H^0(X_0, \sQ \otimes \sK|_{X_0}) = S^0(X_0,  (\sQ \otimes \sK)|_{X_0})
\end{equation}
and
\begin{equation}
\label{eq:relative_canonical_nef2:global_generation}
\sQ \otimes \sK |_{X_y} \textrm{ is globally generated for all } y \in Y .
\end{equation}

Let $Q$ be a divisor of $\sQ$. We prove by induction that $qK_{X/Y} + Q$ is nef for all $q \geq 0$. For $q=0$ the statement is true by the choice of $Q$. Hence, we may assume that we $(q-1)K_{X/Y}  + Q$ is nef. Now, we verify that the conditions of Proposition \ref{prop:nef} hold for $N:=qK_{X/Y} + Q$ and $\sN:= \sO_X(N)$. Indeed:
\begin{itemize}
\item the divisor
\begin{equation*}
 N- K_{X/Y}  = (q-1)K_{X/Y}  + Q
\end{equation*}
is not only $f$-ample, but also nef by the inductional hypothesis,
\item using the $f$-nefness of $K_{X/Y} $ and \eqref{eq:relative_canonical_nef2:assumption_central_fiber},
\begin{equation*}
H^0(X_0, \sN|_{X_0}) = S^0(X_0,  \sN|_{X_0}),
\end{equation*}
\item since all the summands of $N$ are $f$-nef, so is $N$,
\item for every $y \in Y$,  $N|_{X_y}$ is globally generated by  \eqref{eq:relative_canonical_nef2:global_generation}.
\end{itemize}
Hence Proposition \ref{prop:nef} implies that $N$ is nef. This finishes our inductional step, and hence the proof of the nefness of $q K_{X/Y}  + Q$ for every $q \geq 0$. However, then $ K_{X/Y}$ has to be nef as well. This concludes our proof.
\end{proof}

\begin{proof}[Proof of \autoref{thm:semi_positivity}]
By \autoref{thm:relative_canonical_nef2}, we know that $K_{X/Y}$ is not only $f$-ample, but also nef. Further for every $m \gg 0$,
\begin{equation*}
H^0(X_0, \omega_{X/Y}^m|_{X_0})
= \underbrace{H^0(X_0, \omega_{X_0}^m)}_{\textrm{$X$ is Gorenstein over $Y$}}
= \underbrace{S^0(X_0, \omega_{X_0}^m)}_{\textrm{by \autoref{exc:S_0_equals_H_0}}}
= \underbrace{S^0(X_0, \omega_{X/Y}^m|_{X_0})}_{\textrm{$X$ is Gorenstein over $Y$}} .
\end{equation*}
Then \autoref{prop:semi_positive} applies to $N= \omega_{X/Y}^m$ for each $m \gg 0$, which concludes our proof.
\end{proof}

\subsection{Miscellaneous exercises}

\begin{exercise}
\label{exc:Del_Pezzos}
Show that a Del-Pezzo surface $X$ over an algebraically closed field $k$ of characteristic $p>0$ is globally $F$-regular if
\begin{enumerate}
 \item $K_X^2 \geq 4$,
\item $K_X^2 =3 $, and $p > 2$,
\item $K_X^2 =2 $, and $p > 3$ and
\item $K_X^2 =1 $, and $p > 5$.
\end{enumerate}
Further show that in the missing cases  there are examples of both globally $F$-regular and not globally $F$-regular del-Pezzos.

\Hint{Use \autoref{ex.ExtendingSections}.}
\end{exercise}

\begin{exercise}
Let $X$ be a smooth projective variety over $k$. Show that Frobenius stable canonical ring $R_S(X):=\bigoplus_{n \geq 0} S^0(X, \omega_X^m)$ is an ideal of the canonical ring $R(X):=\bigoplus_{n \geq 0} H^0(X, \omega_X^m)$. Further, show that $R_S(X)$ is a birational invariant (of smooth projective varieties over $k$). Then show that it does not change during a run of the MMP, where for singular varieties $R_S(X):=\bigoplus_{n \geq 0} S^0 \left(X, \omega_X^{[m]}\right)$ ($\sF^{[m]}$ denotes the reflexive power, that is, the double dual $(\sF^{\otimes m})^{\vee \vee}$  of the tensor power, see \cite{HartshorneGeneralizedDivisorsOnGorensteinSchemes} for the theory of reflexive sheaves). Deduce then that if the index of the canonical model $X_{\can}$ is coprime to $p$, then for $m$ divisible enough, $S^0(X, \omega_X^m)= H^0 \left(X_{\can}, \sigma(X_{\can}, 0 ) \otimes  \omega_{X_{\can}}^{[m]} \right)$. Give then an example of an $X$ for which $R_S(X)$ is not finitely generated as a ring.

\Hint{To show that $R_S(X)$ is a birational invariant, for two birational varieties $Z$ and $Y$ take a normal variety $W$ that maps with a  birational morphism to both $Z$ and $Y$. For example one can take $W$ to be the normalization of the closure of the graph of the birational equivalence. Finally show that $R_S(Z) = R_S(W) = R_S(Y)$ by showing that $H^0 \left(Z, \omega_Z^m \right) \cong H^0 \left(W, \omega_W^{[m]} \right)$ and $H^0 \left(Z, \omega_Z^{1+(m-1)p^e} \right) \cong H^0 \left(W, \omega_W^{[1+(m-1)p^e]} \right)$. There is one more subtlety: the trace maps also have to be identified, however that is not hard to do because of the open sets where the maps $W \to Z$ and $W \to Y$ are isomorphisms.

Showing that MMP does not change $R_S(X)$ is similar, using that for each step the discrepancies do not decrease. For the final conclusion use \autoref{exc:S_0_equals_H_0}.}
\end{exercise}

\begin{definition}
We define the Frobenius stable Kodaira-dimension $\kappa_S(X)$ of a smooth projective variety $X$ over $k$ to be the growth rate of $\dim_k S^0(X, \omega_X^m)$. That is, it is the integer $d$, such that there are positive real numbers $a$ and $b$ for which
\begin{equation*}
 a m^d <\dim_k S^0(X, \omega_X^m)< b m^d,
\end{equation*}
for every divisible enough $m$. We say $d= - \infty$ if $ S^0(X, \omega_X^m)=0$ for every $m >0$.
\end{definition}

\begin{exercise}
Show that if $\kappa_S(X)\geq0$ or $X$ is of general type then $\kappa_S(X)=\kappa(X)$. Find examples for which $\kappa_S(X)=- \infty$ and $\kappa(X)$ is any number between $0$ and $\dim X -1$.

(Note: solutions can be found in \cite[Section 4]{Hacon_Patakfalvi_A_generic_vanishing_in_positive_characteristic}).
\end{exercise}

\begin{exercise}
\label{exc:explicit_S_0}
Let $X$ be an irreducible hypersurface of degree $d$ in $\bP^n$ ($n \geq 2$) defined by $f(x_0,\dots, x_n)=0$ such that $X \cap D(x_0) \neq \emptyset$ (i.e., $f$ is not a polynomial of $x_0$). Let $\tilde{f}(x_1,\dots,x_n):= f(1,x_1,\dots,x_n)$ and let $P_l$ be the space of polynomials of degree at most $l$ in the variables $x_1, \dots, x_n$. Further, let $\Phi_e$ be the $k$-linear map $k[x_1,\dots,x_n] \to k[x_1,\dots,x_n]$, for which
\begin{equation*}
\Phi_e\left( \prod_{i=1}^n x_i^{j_i} \right) = \left\{
\begin{matrix}
\displaystyle\prod_{i=1}^n x_i^{\frac{j_i- p^e +1}{p^e}} & \textrm{if $p^e | j_i - p^e +1$ for all $i$} \\
0 & \textrm{otherwise}
\end{matrix}
\right. .
\end{equation*}
Consider then
\begin{equation*}
V_e:= \Phi_e\left( \tilde{f}^{p^e-1} \cdot P_{ (d-n-1)( 1+ (m-1) p^e)} \right).
\end{equation*}
That is, $V_e$ is  the image via $\Phi_e$ of the space containing all the polynomials that are obtained by multiplying a degree at most $(d-n-1)( 1+ (m-1) p^e)$ polynomial $p^e-1$ times with $\tilde{f}$. Let $W_e \subseteq V_e$ be the subspace of $V_e$ containing the polynomials divisible by $\tilde{f}$.

Show that then the image of
\begin{equation*}
H^0(X, \omega_X^{m-1} \otimes F^e_* \omega_X ) \to H^0(X, \omega_X^m)
\end{equation*}
can be identified with the quotient $V_e /W_e$. In particular,  $S^0(X, \omega_X^m)$ can be identified with $V_e/W_e$ for $e \gg 0$.

\Hint{Show that there is a commutative diagram as follows.
\begin{equation*}
\xymatrix{
H^0(\bP^n, \sO_{\bP^n}((1 + (m-1)p^e)(d-n-1))) \ar[r] \ar[d]^{\cdot f^{p^e-1}}
& H^0\left(X, \omega_X^{1 + (m-1)p^e} \right)  \ar[dd]^{\tr}
\\ H^0(\bP^n, \sO_{\bP^n}((1 + (m-1)p^e)(d-n-1)+ (p^e-1)d)) \ar[d]^{\tr}
\\ H^0(\bP^n, \sO_{\bP^n}(m(d-n-1))) \ar[r]
& H^0(X, \omega_X^m)
}
\end{equation*}
Then show that the horizontal arrows are surjective and restrict everything to $D(x_0)$.}
\end{exercise}

\begin{exercise}
Let $X$ be the surface in $\bP^3$ defined by $x^5 + y^5 + z^5+v^5$. Compute $S^0(\omega_X)$ for each prime $p$. Further, for $p=2$, show that $S^0(\omega_X)=S^0(\omega_X^2)=0$, but $S^0(\omega_X^3) \neq 0$.
\Hint{use \autoref{exc:explicit_S_0}.}
\end{exercise}

\subsection{Problems}

\begin{question}
Fujita's conjecture states that if $X$ is a smooth projective variety and $L$ an ample line bundle on it, then $K_X + (\dim X +1) L$ is free and $K_X + (\dim X + 2) L $ is very ample. This is not known in positive characteristic even for surfaces. For surfaces there are results when $L$ is special \cite{Ekedahl_Canonical_models_of_surfaces_of_general_type_in_positive_characteristic, Shepherd_Barron_Unstable_vector_bundles_and_linear_systems_on_surfaces_in_characteristic_p,Terakawa_The_d_very_ampleness_on_a_projective_surface_in_positive_characteristic, DiCerboFanelliEffectiveMatsusakaSurfaces}. 
However, the full conjecture is not known even in the surfaces case.
\end{question}

\begin{question}
What is the lowest $m$ such that the semi-positivity of \autoref{thm:semi_positivity}
holds? The question is already interesting if one fixes the dimension, so for example for families of surfaces. Note that for  $m=1$ the semi-positivity is known not to hold even for families of curves by \cite[3.2]{Moret_Bailly_Familles_de_courbes_et_de_varietes_abeliennes_sur_P_1_II_exemples}.
\end{question}

\begin{question}
Compute the semi-stable rank of the Hasse-Witt matrix or equivalently the dimension of $S^0(X, \omega_X)$  of general elements in the components of the moduli space of surfaces. By a conjecture of Grothendieck this is non-zero, so finding a component where this is zero would be interesting. A few related articles: \cite{Liedtke_Algebraic_surfaces_of_general_type_with_small_c_1_square_in_positive_characteristic,Liedtke_Non_classical_Godeaux_surfaces,Liedtke_Uniruled_surfaces_of_general_type,Miranda_Nonclassical_Godeaux_surfaces_in_characteristic_five,Hirokado_Singularities_of_multiplicative_p_closed_vector_fields_and_global_one_forms_of_Zariski_surfaces}.
\end{question}

\subsubsection{Log-Fano Varieties}
\label{subsec:log_Fanos}

\begin{question}
If $I \subseteq (0,1) \cap \bQ$ is a finite subset, then is there a $p_0$ depending only on $I$, such that for all log-Del Pezzos $(X, \Delta)$ for which the coefficients of $\Delta$ are in
\begin{equation*}
D(I):= \left\{ \left. \frac{m + \sum_{j=1}^l a_j i_j}{m+1} \right| m,l \in \bN, a_j \in \bN, i_j \in I  \right\} \cap [0,1]
\end{equation*}
$(X, \Delta)$ is strongly $F$-regular? Note that this is the two dimensional version of \cite[4.1]{Cascini_Gongyo_Schwede_Uniform_bounds_for_strongly_F_regular_surfaces}, and it might be hard in full generality. However, the question is already interesting for $I=\{1\}$ or for $D(I)$ replaced by any smaller set, for example by $\{1\}$. Other fixed $I$'s are also interesting.

The main interest in the question stems from applying it to threefold singularities, as in \cite{Cascini_Gongyo_Schwede_Uniform_bounds_for_strongly_F_regular_surfaces}, or to threefold fibrations the geometric generic fibers of which are log-Del Pezzos (in this case it would yield semi-positivity statements using \autoref{thm:semi_positivity}).  The question is interesting in any higher dimension as well, so for log-Del Pezzo replaced by log-Fano.
\end{question}

\begin{question}
Classify globally $F$-regular smooth Fano threefolds of Picard number 1. According to \cite{Shepherd_Barron_Fano_threefolds_in_positive_characteristic}, the classification in any characteristic agrees with the characteristic zero classification \cite{Iskovskikh_Prokhorov_Fano_varieties}, which is a finite list of deformation equivalence classes. Hence it should be possible to determine a list of globally $F$-regular ones as in \autoref{exc:Del_Pezzos}.
\end{question}

\begin{question}
For the cases when not all the Del-Pezzos are globally $F$-regular (see \autoref{exc:Del_Pezzos}), describe the locus of the non globally $F$-regular ones in the moduli space. That is, describe its dimension. Is it closed?
\end{question}

\begin{question}
In \cite{Mori_Saito_Fano_threefolds_with_wild_conic_bundle_structures} it is hinted to conclude the classification of smooth Fano threefolds along the line of the characteristic zero classification \cite{Iskovskikh_Prokhorov_Fano_varieties,Mori_Mukai_Classification_of_Fano_3_folds_with_Picard_number_at_least_two,Mori_Mukai_Erratum_classification_of_Fano_3_folds_with_Picard_number_at_least_two}. The Picard number one and two cases are done in \cite{Shepherd_Barron_Fano_threefolds_in_positive_characteristic} and  \cite{Saito_Fano_threefolds_with_Picard_number_2_in_positive_characteristic}. Finish the higher Picard number cases.
\end{question}

\subsubsection{MMP}

\begin{question}
Classify the singularities that the geometric general fiber of the Iitaka fibration of 3-folds can have. Recall that the Iitaka fibration is a fibration $f: X \to Y$ such that $\omega_X \cong f^* \sL$ for some big line bundle $\sL$. Here you can assume $X$ to be smooth for the first. In general, it would be interesting to have an answer for $X$ with terminal singularities, which is most likely hard. Note that in dimension two only cusps appear and only in characteristics 2 and 3 (see \cite{Mumford_Enriques_classification_of_surfaces_in_char_p,Bombieri_Mumford_Enriques_classification_of_surfaces_in_char_p_II,Bombieri_Mumford_Enriques_classification_of_surfaces_in_char_p_III,Schutt_Shioda_Elliptic_surfaces} and \cite[Chapter V]{Cossec_Dolgachev_Enriques_surfaces_I}).
\end{question}

\begin{question}
Can one run MMP for threefolds of characteristic 2, 3 or 5? The situation for characteristic $p>5$ has been recently mostly cleared out in a series of papers \cite{Hacon_Xu_On_the_three_dimensional_minimal_model_program_in_positive_characteristic,CasciniTanakaXuOnBPF,Birkar_Existence_of_flips_and_minimal_models_for_3_folds_in_char_p}. If $p \leq 5$, then though some generalized extremal contractions are known to exist \cite[0.5]{Keel_Basepoint_freeness_for_nef_and_big_line_bundle_in_positive_characteristics}, it is not known whether flips exist.
\end{question}

\begin{question}
Prove full cone theorem, and full contraction theorems for threefolds (at least in characteristic $p>5$). The existing statements, though are fantastic achievements, have a few unnatural hypotheses: line bundles are not semi-ample, only endowed with a map \cite[0.5]{Keel_Basepoint_freeness_for_nef_and_big_line_bundle_in_positive_characteristics}, there are finitely many extremal rays that are not known to contain rational curves \cite[0.6]{Keel_Basepoint_freeness_for_nef_and_big_line_bundle_in_positive_characteristics} (see also \cite[5.5.4]{Keel_Basepoint_freeness_for_nef_and_big_line_bundle_in_positive_characteristics}), if $K_X+B$ is not pseudo-effective then only a weak cone theorem is known \cite[1.7]{CasciniTanakaXuOnBPF}. The main question is whether these hypotheses can be removed.
\end{question}

\begin{question}
 Abundance for threefolds?
\end{question}

\begin{question}
MMP for 4-folds?
\end{question}

\subsubsection{Surface and threefold inequalities}

For minimal  surfaces of general type a few inequalities govern the possible values of standard numerical invariants (e.g., $K_X^2\geq 1$, $\chi(\sO_X) \geq 1$, $K_X^2 \geq 2 p_g -4$ (Noether inequality), $K_X^2 \leq 9 \chi(\sO_X)$, c.f., \cite[Section 1.2]{Bauer_Catanese_Pignatelli_Complex_surfaces_of_general_type_some_recent_progress}). Many of these are known either to hold, or there is a good understanding of when they fail in positive characteristic \cite[Sections 8.2, 8.3, 8.4]{Liedtke_Algebraic_surfaces_in_positive_characteristic}. There are a few questions concerning surfaces left and very little is known for threefolds.

\begin{question}
Is there a smooth surface of general type with $\chi(\sO_X) \leq 0$? Note that the possibilities are quite restricted (see \cite[Theorem 8]{Shepher_Barron_Geography_for_surfaces_of_general_type_in_positive_characteristic}).
\end{question}

\begin{question}
Minimal Gorenstein threefolds of general type are known to have $\chi(\sO_X)>0$ in characteristic zero. Is this true in positive characteristic? If not, what are the exceptions? The expectations are statements as \cite[8.4,8.5]{Liedtke_Algebraic_surfaces_in_positive_characteristic}.
\end{question}

\begin{question}
If $X$ is a minimal threefold of general type then is there a lower bound for $K_X^3$ in terms of a (linear) function of the geometric genus $\rho_g(X)$? Note that in characteristic zero this has been shown in \cite{Kobayashi_On_Noether_s_inequality_for_threefolds}, and the analogous statement is known for surfaces of positive characteristic \cite{Liedtke_Algebraic_surfaces_of_general_type_with_small_c_1_square_in_positive_characteristic}.
\end{question}

\section{Seshadri constants, $F$-pure centers and test ideals}
\label{sec:SeshadriEtc}
In \autoref{sec:TraceOfFrob} we introduced $S^0$ and showed how global sections can be extended from them via adjunction along divisors.  At the start of this section we discuss other ways to produce sections in $S^0$.  First we recall Seshadri constants and $F$-pure centers and how to use them to produce sections in $S^0$.  Then, building upon and generalizing the definition of $F$-pure centers we introduce test ideals and explore a number of open questions about them.

We begin with Seshadri constants.

\subsection{Seshadri constants}

Recall the following definition originally found in \cite{DemaillyANumericalCriterionForVeryAmple}.

\begin{definition}[Seshadri constants] \cite{LazarsfeldPositivity1}
Suppose $X$ is a projective variety and $L$ is an ample (or big and nef) line bundle.  Choose $z \in X$ a smooth closed point, let $\pi : Y \to X$ be the blowup of $z$ with exceptional divisor $E$.  We define the Seshadri constant of $L$ at $z$ as
\[
\varepsilon(L, z) = \sup\{t > 0\;|\; \pi^* L - tE \text{ is nef} \}.
\]
\end{definition}

The Seshadri constant is a local measure of positivity of $L$ at $z$.  Its usefulness comes from the following theorem (and variants).

\begin{theorem}
If $X$ is a projective variety over $\bC$ and if $\varepsilon(L, z) > \dim X$ for some smooth closed point $z \in X$, then $\omega_X \otimes L$ is globally generated at $z$.
\end{theorem}
\begin{proof}
Let $n = \dim X$ so that $\pi^* L - nE$ is big and nef.  Consider the following diagram:
    \[
    {\scriptsize
    \xymatrix@C=10pt{
    0 \ar[r] & H^0(X, \bm_z \otimes \omega_X \otimes L) \ar[r] & H^0(X, \omega_X \otimes L) \ar[r]^{\alpha} & H^0(X,\omega_X \otimes L / \bm_z) \ar[r] &H^1(X, \bm_z \otimes \omega_X \otimes L)\\
    0 \ar[r] & H^0(Y, \omega_Y(\pi^* L - nE)) \ar[u] \ar[r] & H^0(Y, \omega_Y(\pi^* L - (n-1)E)) \ar[u]^{\sim} \ar[r] & H^0(E, \omega_E(\pi^* L - n E)) \ar[u]^{\sim} \ar[r] & H^1(Y, \omega_Y(\pi^* L - nE)) \ar[u] \\
    }
    }
    \]
It is an exercise left to the reader to verify the isomorphisms indicated.  By Kawamata-Viehweg vanishing $H^1(Y, \omega_Y(\pi^* L - nE)) = 0$ and hence $\alpha$ surjects proving that $\omega_X \tensor L$ is globally generated at $z$.
\end{proof}

\begin{exercise}
Work in characteristic $p > 0$ and assume that $L$ is ample.  Show that in fact that if $\varepsilon(L, z) > \dim X$ then $S^0(X, \omega_X \otimes L)$ has a section which does not vanish at $z$.
\vskip 3pt
\Hint{Add another row to the diagram with vertical up-pointing arrows induced by Frobenius.  Replace then Kawamata-Viehweg vanishing with Serre vanishing.}
\end{exercise}

\begin{question}[Perhaps hard]
\label{quest.LowerBoundSeshadri}
Are there lower bounds for Seshadri constants at very general points in characteristic $p > 0$ ala \cite{EinKuchleLazarsfeldLocalPositivitys}?
\end{question}

There is another version of Seshadri constants which has recently been studied in characteristic $p > 0$ \cite{Mustata_Schwede_A_Frobenius_variant_of_Seshadri_constants}.  This definition is inspired by the characterization of ordinary Seshadri constants described by separation of jets \cite[Chapter 5]{LazarsfeldPositivity1}.

Given a positive integer $e$, we say that a line bundle $L$ on $X$ \emph{separates
$p^e$-Frobenius jets at $z$} if the restriction map
\begin{equation}\label{restriction_Frobenius}
H^0(X,L)\to H^0(X,L\otimes\cO_X/\bm_z^{[p^e]})
\end{equation}
is surjective (here $\bm_z^{[p^e]}$ is the ideal generated by the $p^e$th powers of the elements of $\bm_z$).

Let $s_F(L^m, z)$ be the largest $e\geq 1$ such that $L^m$ separates
$p^e$-Frobenius jets at $z$ (if there is no such $e$, then we put
$s_F(L^m, Z)=0$).

\begin{definition}[$F$-Seshadri constants] \cite{Mustata_Schwede_A_Frobenius_variant_of_Seshadri_constants}
Let $L$ be ample, the \emph{Frobenius-Seshadri constant of $L$ at $z$} is
$$\epsilon_F(L, z):=\sup_{m\geq 1} \frac{p^{s_F(L^m; z)}-1}{m}.$$
\end{definition}

It is not difficult to show that $\frac{\epsilon(L,z)}{n}\leq \epsilon_F(L, z)\leq\epsilon(L, z)$, see \cite[Proposition 2.12]{Mustata_Schwede_A_Frobenius_variant_of_Seshadri_constants}.  Hence \autoref{quest.LowerBoundSeshadri} is equivalent to:

\begin{question}
Are there lower bounds for $F$-Seshadri constants at very general points in characteristic $p > 0$?
\end{question}

We also have

\begin{theorem}
If $\epsilon_F(L, z) > 1$ then $S^0(X, \sigma(X, 0) \otimes \O_X(L + K_X))$ globally generates $\omega_X \otimes L$ at $z$.
\end{theorem}

\begin{exercise}
Use the above theorem to give another proof of the fact that if $\epsilon(L, z) > \dim X$ then $S^0(X, \sigma(X, 0) \otimes \O_X(L + K_X))$ globally generates $\omega_X \otimes L$ at $z$.
\end{exercise}

Perhaps a more approachable question (inspired by the behavior of the usual Seshadri constant) is:

\begin{question}
If $L_1$ and $L_2$ are ample divisors and $z \in X$ is a smooth closed point then show that:
$$\epsilon_F(L_1\otimes L_2; z)\geq\epsilon_F(L_1; z)+\epsilon_F(L_2; z).$$
\end{question}

Since the Frobenius Seshadri constant is distinct from the ordinary Seshadri constant, another interesting question is:

\begin{question}
What does the $F$-Seshadri constant correspond to in characteristic zero?  Is there a geometric characteristic zero definition that behaves in the same way philosophically?  (The toric case for torus invariant points is handled at the end of \cite{Mustata_Schwede_A_Frobenius_variant_of_Seshadri_constants}).
\end{question}

\subsection{$F$-pure centers}

Now we move on to $F$-pure centers.  If the reader is introduced to $S^0$, one exercise to keep in mind is \autoref{ex.SurjectionOfS0ForCenters} which generalizes \autoref{ex.ExtendingSections} from the case of divisors to higher codimension $F$-pure centers.

\begin{definition}
Suppose $(X, \Delta)$ is a pair with $\Delta \geq 0$ an effective $\bQ$-divisor such that $(p^e - 1)(K_X + \Delta)$ is Cartier.  Let $\phi_{e, \Delta} : F^e_* \sL \to \O_X$ be the corresponding map.  We say that a subvariety $Z \subseteq X$ is \emph{an $F$-pure center of $(X, \Delta)$} if
\begin{itemize}
\item{} $\phi_{e, \Delta}(F^e_* (I_Z \cdot \sL)) \subseteq I_Z$ and
\item{} $(X, \Delta)$ is $F$-pure (ie $\phi_{e, \Delta}$ is surjective) at the generic point of $Z$.
\end{itemize}
\end{definition}

\begin{remark}
In the case that $\phi_{e, \Delta} : F^e_* \O_X \to \O_X$ is a splitting of Frobenius\footnote{Where again $\Delta \geq 0$ is necessarily an effective $\bQ$-divisor such that $(p^e - 1)(K_X + \Delta)$ is Cartier, as above.}, the $F$-pure centers are called the \emph{compatibly split subvarieties of $\phi_{e, \Delta}$}.
\end{remark}

\begin{exercise}
With $(X, \Delta)$ as above, show that $Z$ is an $F$-pure center if and only if
\begin{itemize}
\item{} for every Cartier divisor $H$ containing $Z$ and every $\varepsilon > 0$, $(X, \Delta + \varepsilon H)$ is not $F$-pure and
\item{} $(X, \Delta)$ is $F$-pure at the generic point of $Z$.
\end{itemize}
If you know the definition of a log canonical center, prove that the analogous characterization can be used to define log canonical centers.
\end{exercise}

Let's do an example.  First we state (but don't prove) a variant of Fedder's criterion for $F$-pure centers.

\begin{lemma}
Suppose that $S = k[x_1, \ldots, x_n]$ and $R = S/I$ and $X = \Spec R$.  If $K_X$ is $\bQ$-Cartier with index not divisible by $p > 0$, then, after localizing further if necessary, $I^{[p^e]} : I = \langle g_e \rangle + I^{[p^e]}$ for some single polynomial $g_e$ (depending on $e$).  Then $V(Q) = Z \subseteq X$ is an $F$-pure center of $(X, 0)$ if and only if $g_e \in Q^{[p^e]} : Q$ and $g_e \notin Q^{[p^e]}$.
\end{lemma}

\begin{exercise}
Suppose that $X = \Spec k[x_1, \ldots, x_n] / \langle x_1 x_2 \cdots x_n\rangle$.  Identify the $F$-pure centers of $(X, 0)$.
\end{exercise}

Let's prove inversion of $F$-adjunction.

\begin{exercise}[Inversion of $F$-adjunction]
Suppose $(X, \Delta)$ is a pair and $Z \subseteq X$ is an $F$-pure center.  For simplicity assume that $Z$ is normal.  Show that the map $\phi_{e, \Delta} : F^e_* \sL \to \O_X$ induces a map $\phi_Z : F^e_* \sL|_Z \to \O_Z$ and hence an effective $\bQ$-divisor $\Delta_Z$ (this wasn't covered in the lecture but see the notes above).  Show that $(K_X + \Delta)|_Z \sim_{\bQ} K_Z + \Delta_Z$.  Further show that $(X, \Delta)$ if $F$-pure in a neighborhood of $Z$ if and only if $(Z, \Delta_Z)$ is $F$-pure.
\end{exercise}


It is not hard to show that $Z \subseteq X$ is a log canonical center of $(X, \Delta)$ then $Z$ is also an $F$-pure center as long as $(X, Z)$ is $F$-pure at the generic point of $Z$ (using the fact that $F$-pure singularities are log canonical \cite{HaraWatanabeFRegFPure}).

\begin{exercise}
Prove the above assertion.
\end{exercise}

Based on recent work \cite{MustataSrinivasOrdinary,MustataOrdinary2} consider the following conjecture.

\begin{conjecture}[Weak ordinarity]
Suppose that $X$ is a smooth projective variety over $\bC$, $A = \Spec R$ where $R$ is a finitely generated $\bZ$-algebra containing $\bZ$ and $X_A \to A$ is a spreading out of $X$ over $A$ (so that $X_A \times_A \bC \cong X$), in other words a family of characteristic $p > 0$ models of $X$.  Then there exists a Zariski dense set of closed points $U \subseteq A$ such that for every point $p \in U$ we have that if $X_p = X_A \times_A k(p)$ then the Frobenius morphism:
\[
H^i(X, \O_X) \to H^i(X, F^e_* \O_X)
\]
is bijective.
\end{conjecture}

In \cite{TakagiAdjointIdealsAndACorrespondence}, it was shown that this conjecture implies that if $(X, \Delta)$ is log canonical and $(X_A, \Delta_A)$ is a family of characteristic $p > 0$ models of $(X, \Delta)$ as above, then $(X_p, \Delta_p)$ is $F$-pure for a Zariski dense set of $p \in A$.  Hence we are inspired to ask:

\begin{question}
Assume the weak ordinarity conjecture and suppose that $\{Z_i\}$ is the (finite) set of log canonical centers of a log canonical pair $(X, \Delta)$.  Is it true that for any $(X_A, \Delta_A) \to A$ and $\{ (Z_i)_A \} \to A$ as above, there exists a Zariski-dense set of closed points $p \in A$ such that $\{ (Z_i)_p \}$ is exactly the set of $F$-pure centers of $(X_p, \Delta_p)$?
\vskip 3pt
\emph{Comment:} It is easy to see that the $(Z_i)_p$ are $F$-pure centers (exercise), but it may be more difficult to see that this is all of them.
\end{question}

Finally, we mention that $F$-pure centers can be used to construct global sections.

\begin{exercise}
\label{ex.SurjectionOfS0ForCenters}
Formulate and prove a generalization of \autoref{ex.ExtendingSections}.  In particular, show that $S^0(X, \sigma(X, \Delta) \tensor \O_X(M))$ surjects onto $S^0(Z, \sigma(Z, \Delta_Z) \otimes \O_X(M))$ under appropriate hypotheses (here $Z$ will be an $F$-pure center of $(X, \Delta)$).
\end{exercise}

The question is then how to construct $F$-pure centers.  The common techniques involving sections of high multiplicity (\cf for example \cite[Chapter 10]{LazarsfeldPositivity2}) seem to not work since the $F$-pure threshold can have a $p$ in the denominator (this actually gets quite subtle, simple perturbation techniques seem not to work).  For some ways around this issue see the recent work of \cite{CasciniTanakaXuOnBPF}.

\subsection{Test ideals}

In this subsection we introduce test ideals, a characteristic $p > 0$ analog of multiplier ideals.

Suppose $X$ is integral and normal and $\Delta$ is a $\bQ$-divisor.  As before, we have observed that if the index of $K_X + \Delta$ is not divisible by $p$, then we obtain a map $\phi_{\Delta} : F^e_* \sL \to \O_X$.  In this setting we define:

\begin{definition}
With notation as above, we define the \emph{test ideal} $\tau(X, \Delta)$ to be the smallest non-zero ideal sheaf $J$ such that $\phi_{\Delta}(F^e_* (J \cdot \sL)) \subseteq J$.
\end{definition}

We give a definition in the local setting $X = \Spec R$ but when $K_X + \Delta$ has no $\bQ$-Cartier requirements.

\begin{definition}
\label{def.TauDefinition}
Assume $X = \Spec R$ is integral and normal and $\Delta\geq 0$ is a $\bR$-divisor.  Then we define the \emph{test ideal} $\tau(X, \Delta)$ to be the smallest non-zero ideal sheaf $J$ such that for all $e \geq 0$ and all $\phi \in \Hom_R(F^e_* \O_X(\lceil (p^e - 1) \Delta \rceil), \O_X)$ we have $\phi(F^e_* J) \subseteq J$.
\end{definition}

It is probably worth remarking that $\Hom_R(F^e_* \O_X(\lceil (p^e - 1) \Delta \rceil), \O_X)$ consists precisely of those $\phi'$ corresponding to divisors $\Delta'$ such that $(p^e - 1)(K_X + \Delta') \sim 0$ and such that $\Delta' \geq \Delta$.

Since in either definition, $\tau(X, \Delta)$ is the \emph{smallest} such ideal, it is not obvious that the test ideal exists.  To show it exists (locally), you first find a $c \in R$ such that for every $\ba \subseteq R$ an ideal, we have $c \in \phi(F^e_* \ba)$ for some $\phi \in \Hom_R(F^e_* \O_X(\lceil (p^e - 1) \Delta \rceil), \O_X)$.  This is tricky, and such elements are called \emph{test elements}.  However, once the $c$ is found constructing $\tau(X, \Delta)$ is simply
\[
\tau(X, \Delta) = \sum_{e \geq 0} \sum_{\phi} \phi(F^e_* (c \cdot \O_X))
\]
in the second case (which subsumes the first) where $\phi$ varies over $\Hom_R(F^e_* \O_X(\lceil (p^e - 1) \Delta \rceil), \O_X)$.  This is clearly the smallest ideal satisfying the condition of \autoref{def.TauDefinition} and containing $c$.  But since any ideal satisfying \autoref{def.TauDefinition} obviously contains $J$, we have constructed $\tau(X, \Delta)$.

\begin{exercise}
Show that the formation $\tau(X, \Delta)$ of commutes with localization and so our local definition glues to a global one.
\end{exercise}

\begin{remark}
It is natural to think of $\tau$ as a robust version of $\sigma$ in the same way that the augmented base locus is a robust version of the stable base locus.  In fact, there is a global version of $S^0$ which builds in a ``test element''.   Assume that $(X, \Delta)$ is a pair and that $(p^n-1)(K_X + \Delta)$ is Cartier (where $n$ is the smallest positive integer with that property), set $\sL_e = \O_X( (1-p^e)(K_X + \Delta))$ so that we have a map $\phi_{\Delta}^e : F^e_* \sL_e \to \O_X$ for each $e$ divisible by $n$.  For any Cartier divisor $M$ define $P^0(X, \tau(X, \Delta) \otimes \O_X(M))$ to be
\[
\bigcap_{0 \leq D \subseteq X} \bigcap_{e_{0}\geq 0}  \left( \sum_{e \geq e_0} \Tr_{F^e} \Big(H^0\big(X, \Big( F^e_* \O_X(\lceil K_X - p^e (K_X + \Delta) + p^e M  -  D \rceil)  \big)\Big)\Big)\right)
\]
where the intersection runs over all effective Weil divisors $\geq 0$ on $X$.
\begin{exercise}
For simplicity, assume that $(p-1)(K_X + \Delta)$ so that there are no roundings.  Then show that for $P^0 \subseteq S^0$ and further that the image of the maps used to define $P^0$ at the sheaf level is simply $\tau(X, \Delta) \otimes M$
\end{exercise}
\end{remark}

One of the main reasons people have been interested in test ideals is that if $(X_{\bC}, \Delta_{\bC})$ is a log $\bQ$-Gorenstein pair defined over $\bC$ and $X_A \to A$ is a family of characteristic $p$ models, then $\tau(X_p, \Delta_p) = \big(\mJ(X_{\bC}, \Delta_{\bC})\big)_{p}$ for all closed points $p$ in an open dense set of $A$.  In other words, the multiplier ideal reduces to the test ideal for $p \gg 0$ (see \cite{TakagiInterpretationOfMultiplierIdeals}).  Note that if the coefficients of $\Delta$ vary, then this is not true.  It is conjectured then that for $H$ an effective Cartier divisor then $\tau(X_p, \Delta_p+tH) = \big(\mJ(X_{\bC}, \Delta_{\bC})+tH\big)_{p}$ for all $t$ and for all $p$ is a Zariski dense set of points of $A$ but this is related to some very hard conjectures, see \cite{MustataSrinivasOrdinary,MustataOrdinary2,BhattSchwedeTakagiWeakOrdinarity} for discussion.

It would be natural to try to generalize the results on the behavior of test ideals to the non-$\bQ$-Gorenstein setting.

\begin{question}
Suppose that $(X_{\bC}, \Delta_{\bC})$ is a pair but that $K_{X_{\bC}} + \Delta_{\bC}$ is not assumed to be $\bQ$-Cartier.  Is it true that $\tau(X_p, \Delta_p) = \big(\mJ(X_{\bC}, \Delta_{\bC})\big)_{p}$ for all closed points $p$ in an open and Zariski-dense set of $A$ (where $X_A \to A$ is as above)?
\end{question}

\subsection{Jumping numbers}

Suppose that $(X, \Delta)$ is as above and that $H$ is a Cartier divisor.  Consider the behavior of $\tau(X, \Delta+tH)$ as $t \in \bR_{\geq 0}$ varies.  Since we have asserted that test ideals are characteristic $p > 0$ analogs of multiplier ideals, one should expect that $\tau(X, \Delta+tH)$ jumps (as $t$ varies) at a set of rational $t \in \bQ_{\geq 0}$ without limit points, at least assuming that $K_X + \Delta$ is Cartier.  Indeed, this is the case (in the log-$\bQ$-Gorenstein setting).

First an easy exercise (you can assume the existence of test elements):
\begin{exercise}
If $t \geq t'$ show that $\tau(X, \Delta + tH) \subseteq \tau(X, \Delta + t'H)$.
\end{exercise}

\begin{definition}
An \emph{$F$-jumping number of $(X, \Delta)$ with respect to $H$} is a real number $t \geq 0$ such that $\tau(X, \Delta + tH) \neq \tau(X, \Delta+(t-\varepsilon) H)$ for any $\varepsilon > 0$.
\end{definition}

First you shall prove that this notion is sensible.

\begin{exercise}
Suppose that $t \geq 0$ is a rational number.  Prove that there exists an $\varepsilon > 0$ such that $\tau(X, \Delta+(t+\varepsilon)H) = \tau(X, \Delta + tH)$.
\vskip 3pt
\Hint{Absorb the bigger $t$ into the test element somehow and use the fact that the test ideal can be written as a finite sum.}
\end{exercise}

\begin{question}
Assume now that $(X, \Delta)$ is a pair and that $K_X + \Delta$ is not necessarily $\bQ$-Cartier.
Does the set of $F$-jumping numbers of $(X, \Delta)$ with respect to $H$ have any limit points?  Are the $F$-jumping numbers rational?
\end{question}

One would assume that the $F$-jumping numbers are not necessarily rational based upon the situation for the multiplier ideal \cite{UrbinatiDiscrepanciesOfNonQGorensteinVars}.  However, one might hope that the $F$-jumping numbers do not have limit points (as is also hoped in characteristic zero).  For some partial progress, see \cite{TakagiTakahashiDModulesOverRingsWithFFRT,BlickleTestIdealsViaAlgebras,KatzmanSchwedeSinghZhang}.  A first target might be to try to consider the case of an isolated singularity with $\Delta = 0$.  Some relevant methods might be found for instance in \cite{LyubeznikSmithCommutationOfTestIdealWithLocalization}.

\subsection{Test ideals of non-$\bQ$-Gorenstein rings}
\label{subsec.TestIdealsInNonQGor}

\begin{question}
Suppose that $(X, \Delta)$ is a pair and that $K_X + \Delta$ is not assumed to be $\bQ$-Gorenstein.  Then does there exist a $\Delta' \geq \Delta$ such that $K_X + \Delta'$ is $\bQ$-Cartier (with index not divisible by $p$ hopefully) such that $\tau(X, \Delta) = \tau(X, \Delta')$.
\end{question}

The answer to this is known to be yes if $\tau(X, \Delta)= \O_X$ at least in the affine setting, see \cite{Schwede_Smith_globally_F_regular_an_log_Fano_varieties,SchwedeTestIdealsInNonQGor}.  More generally, it is also known that there exist finitely many $\Delta_i \geq \Delta$ such that $\tau(X, \Delta) = \sum_i \tau(X, \Delta_i)$ and such that each $(K_X + \Delta_i)$ is $\bQ$-Cartier with index not divisible by $p > 0$, see \cite{SchwedeTestIdealsInNonQGor}.  As a general strategy on this problem, one would hope that some general choice of $\Delta'$ would work but more work is required.  In characteristic zero the analog is true if one defines $\mJ(X, \Delta)$ is the non-log-$\bQ$-Gorenstein setting as in \cite{DeFernexHaconSingsOnNormal}.

\begin{exercise}
Suppose that $R$ is a local ring and that $X = \Spec R$.  Suppose that we know that $\tau(X, \Delta) = \sum_i \tau(X, \Delta_i)$ with $K_X + \Delta_i$ $\bQ$-Cartier as before and also that $\tau(X, \Delta) = \O_X$.  Prove that $\tau(X, \Delta_i) = \tau(X, \Delta)$ for some $i$.
\end{exercise}

Here $\mJ(X_{\bC}, \Delta_{\bC})$ is defined as in \cite{DeFernexHaconSingsOnNormal}.  Some recent work on this question was done in \cite{DeFernexDocampoTakagiTuckerComparingMultTest}.  It is easily seen to be true in the toric case \cite{BlickleMultiplierIdealsAndModulesOnToric}.  Perhaps isolated or even isolated graded singularities would be a natural place to start (see \cite{LyubeznikSmithStrongWeakFregularityEquivalentforGraded} for some ideas using commutative algebra language).

\begin{exercise}
Show that the containment $\big(\mJ(X_{\bC}, \Delta_{\bC})\big)_{p} \subseteq \tau(X_p, \Delta_p)$ holds.  \vskip 3pt
\Hint{Use the fact that we know it holds in the $\bQ$-Gorenstein setting.}
\end{exercise}







\section{Numerical Invariants}

\newcommand{\Fpt}{\mathrm{fpt}}
\begin{definition}
  Assume $X = \Spec R$ is integral and normal and $\Delta \geq 0$ is
  an $\R$-divisor.  Suppose that $(X, \Delta)$ is sharply $F$-pure and that $H$ is a
  Cartier divisor.  The $F$-pure threshold of $(X,\Delta)$ along $H$
  is
$
\Fpt((X,\Delta),H) = \sup \{ t \in \R  \; | \; (X, \Delta + tH) \mbox{ is sharply $F$-pure}\}.
$
When $\Delta = 0$, then we write $\Fpt(X, H)$ instead of $\Fpt( (X, 0), H)$.
\end{definition}

Of course, one should compare the definition about with that of the
Log Canonical Threshold (LCT) in characteristic zero (which simply replaces
sharply $F$-pure with Log Canonical throughout).

\begin{exercise}
  When $(X, \Delta)$ is strongly $F$-regular, show that
  $\Fpt((X,\Delta),H)$ is the smallest jumping number of $(X,\Delta)$
  with respect to $H$.
\end{exercise}

It is known that if $K_X + \Delta$ is $\bQ$-Cartier, then $\Fpt((X, \Delta), H) \in \bQ$ (indeed all jumping numbers are rational).  But in general, it is an open question, compare with \cite{UrbinatiDiscrepanciesOfNonQGorensteinVars}.

\begin{problem}
  Find $(X, \Delta)$ and $H$ so that $\Fpt((X,\Delta),H) \not\in
  \Q$ (or show this can't happen).
\end{problem}

It follows easily from \autoref{exc:SFP_implies_lc} that if $K_X + \Delta$ is $\bQ$-Cartier and $(X, \Delta)$ is sharply $F$-pure, then $(X, \Delta)$ is log canonical.  Hence in a fixed characteristic $\Fpt( (X, \Delta), H) \geq \lct( (X, \Delta), H)$.  As we vary the characteristic however we do get some more subtle behavior.

\begin{exercise}
Working in characteristic $p > 0$, let $X=\bA^2$ and $H = V(y^2-x^3)$.  Show that the following formula for $\fpt(X, H)$ holds.
\[
\fpt(X, H) = \left\{ \begin{array}{cl} {1 \over 2} & \text{if $p = 2$} \\ \\{2 \over 3} & \text{if $p = 3$} \\ \\{5 \over 6} & \text{if $p \equiv 1 \mod 6$} \\ \\ {5 \over 6} - {1 \over 6p} & \text{ if $p \equiv 5 \mod 6$} \end{array} \right.
\]
\end{exercise}

\begin{exercise}
If $X = \bA^2$ and $H = V(xy(x+y))$.  Find the formula for $\fpt(X,H)$ in terms of the characteristic.
\end{exercise}

\begin{theorem}
\label{thm.FPTLimitsToLCT}
Suppose $(X_{\bC}, \Delta_{\bC})$ is KLT and $H_{\bC}$ is a Cartier divisor all defined over $\bC$. Take a model $\sX \to \Spec(A)$
  of $X$ (together with $\Delta$ and $H$) over a
  finitely generated $\Z$-algebra domain $A$ (so that $\sX_{0}
  \otimes_{\Frac(A)} \bC = X_{\bC}$).  Then for every $\varepsilon > 0$, there exists a dense open set $U \subseteq \Spec A$, so that $$\lct((X_{\bC}, \Delta_{\bC}), H_{\bC}) - \varepsilon < \Fpt( (X_{\bm}, \Delta_{\bm}), H_{\bm})$$ for all maximal ideals $\bm \in \Spec A$.  Informally this says that
  \[
  \lim_{p \to \infty} \Fpt ( (X_p, \Delta_p), H_p) = \lct((X_{\bC}, \Delta_{\bc}), H_{\bC}).
  \]
\end{theorem}
\begin{exercise}
Use \autoref{thm:klt_SFR_reduction} to prove \autoref{thm.FPTLimitsToLCT}.
\end{exercise}

While we do have that limit, in many examples it appears the two invariants frequently agree on the nose, in other words that $\Fpt( (X_{\bm}, \Delta_{\bm}), H_{\bm}) = \lct( (X_{\bC}, \Delta_{\bC}), H_{\bC})$ for a Zariski dense set of closed points $\bm \in \Spec A$.  Hence we have the following very important conjecture.

\begin{conjecture}
  Suppose that $X_{\bC}$ is a normal affine algebraic variety over $\bC$, and
  $\Delta \geq 0$ is an $\R$-divisor, and $H$ is a Cartier divisor.  Take a model $\sX \to \Spec(A)$
  of $X$ (together with $\Delta$ and $H$) over a
  finitely generated $\Z$-algebra domain $A$ (so that $\sX_{0}
  \otimes_{\Frac(A)} \bC = X_{\bC}$).  Then for a dense set of maximal
  ideals $\bm \in \Spec A$, we have that $\Fpt((\sX_{\bm},
  \Delta_{\bm}),H_{\bm})$ equals $\lct((X_{\bC},\Delta), H)$.
\end{conjecture}

\begin{exercise}
  If $(X,\Delta)$ is strongly $F$-regular and $H$ is a Cartier divisor
  so that $c = \Fpt((X,\Delta),H) \in \Q$ so that $K_{X}+ \Delta +cH$
  has index relatively prime to $p$, show that $(X, \Delta + cH)$ is
  sharply $F$-pure and $\tau(X, \Delta + cH)$
  is a radical ideal.
\end{exercise}

With notation as above, if $K_X + \Delta + cH$ has index divisible by $p$, then it can happen that $\tau(X, \Delta + cH)$ is not radical \cite{MustataYoshidaTestIdealVsMultiplierIdeals}.

\begin{question}
  Does the FPT satisfy analogous ACC properties to the LCT?
\end{question}

\begin{exercise}
  Suppose $X$ is normal and $K_{X}$ is Cartier.   Let $H$ be a Cartier
  divisor.  If $\xi$ is a jumping number of the test ideals $\tau(X, t
  H)$ for $t \in \R_{\geq 0}$, show also that $p \xi$ is a jumping
  number.
  \Hint{consider the image of $\tau(X, ptH)$ under the trace
  map.  Also, using the ideas of the following section, do this again and show that
  $p^{e}\xi$ is a jumping number of the ideals $\tau(X, \Delta + tH)$
  whenever $(p^{e}-1)(K_{X}+\Delta)$ is Cartier.}
\end{exercise}

\begin{definition}
  Suppose that $X = \Spec R$ where $R = S / I$ is a quotient of a
  polynomial ring $S = k[x_{1}, \ldots, x_{n}]$.  Let $R_{d}$ be the
  image the polynomials of degree less than or equal to $d$ in $S$.
  Define a function $\delta : R \to \Z_{\geq 0} \cup \{ - \infty \}$
  by $\delta(f) = d$ whenever $f \in R_{d} \setminus R_{d-1}$.  We say
  that $\delta$ is a gauge for $R$.  It is easy to see that the
  following properties hold.
  \begin{enumerate}
  \item $\delta(f) = - \infty$ if and only if $f = 0$.
  \item Each $R_{d}$ is a finite dimensional $k$ vector space.
    \item $\bigcup_{d}R_{d} = R$.
      \item $\delta(f + g) \leq \max\{ \delta(f), \delta(g)\}$.
\item $\delta(fg) \leq \delta(f) + \delta(g)$.
  \end{enumerate}
  \end{definition}

In the exercises below, we will make use of the notation from the definition.

  \begin{exercise}
    Suppose that $\phi : F^{e}_{*}R \to R$ is an $R$-linear map.  Show
    that there exists a constant $K$ so that
\[
\delta(\phi(F^{e}_{*}f)) \leq \delta(f) / p^{e} + K / p^{e}
\]
for all $f \in R$.  If $\phi$ corresponds to a divisor $\Delta$ on $X
= \Spec R$, show how to use this to get a bound on the degrees of
generators for the test ideals $\tau(X, \Delta + t\Div(g))$ when $g
\in R$ in terms of $\delta(g)$.  See the survey \cite{BlickleSchwedeSurveyPMinusE}
for more details.
  \end{exercise}

  \begin{exercise}
  Use the previous exercise to conclude that the set of $F$-jumping numbers of $\tau(X, \Delta + t \Div(g))$ are a discrete set of rational numbers.
  \end{exercise}

\section{More on test ideals and $F$-Singularities in Families}

\subsection{Bertini theorems for test ideals}

In \autoref{subsec.TestIdealsInNonQGor} we thought about choosing general boundaries $\Delta_i$.  The choice of a general $\Delta_i$ also seems related to the following other problem.

\begin{question}[Bertini's theorem for test ideals]
Suppose that $X$ is a quasi-projective variety and that $(X, \Delta)$ is a pair with $K_X + \Delta$ even $\bQ$-Cartier with index not divisible by $p > 0$.  Suppose that $H$ is a general hyperplane section of $X$.  Is it true that $\tau(X, \Delta) \cdot \O_H = \tau(H, \Delta|_H)$?
\end{question}

Note that the containment $\supseteq$ was shown in \cite[Proposition 2.12(1)]{TakagiAdjointIdealsAndACorrespondence} under some assumptions.
The other direction seems to be hard.  The main approach seems to be to show the following.

 \begin{question}
 \label{quest.A1ForTau}
 If $f : U \to V$ is a flat family with regular (but not necessarily smooth) fibers then is $\tau(V, \Delta) \cdot \O_U \subseteq \tau(U, f^* \Delta)$?  This is known in the case that $\tau(V, \Delta) = \O_V$, in other words where $(V, \Delta)$ is strongly $F$-regular, see \cite{SchwedeZhangBertiniTheoremsForFSings} and \cite[Theorem 7.3]{HochsterHunekeFRegularityTestElementsBaseChange}.  It would be an interesting problem to study on its own.
 \end{question}

We need one other property.

\begin{question}
\label{quest.A2ForTau}
suppose $\phi : Y \to S$ is a morphism of finite type between schemes essentially of finite type over a field $k = \overline{k}$ and that $\Delta \geq 0$ is a $\bQ$-divisor on $Y$ such that $K_Y + \Delta$ is $\bQ$-Cartier with index not divisible by $p$.  If $\overline{g} \in \O_{Y_s}$ is such that $\overline{g} \in \tau(Y_s, \Delta_s)$ for some geometric point $s \in S$, is it true that $\Image(g) \in \tau(Y_{t}, \Delta_t)$ for some $g \in \O_{Y \times_S T}$ (some base change) restricting to $\overline{g}$ and all (geometric) $t$ in a neighborhood of $s$?
\end{question}

Perhaps something like this result can be obtained using the method of \cite{Patakfalvi_Schwede_Zhang_F_singularities_in_families}.  We sketch the main idea for how to use this to get Bertini's theorem.
Suppose that solutions to both these interim problems were positive.  Now consider the following general strategy taken from \cite{CuminoGrecoManaresiAxiomatic}.

Let $Z$ be the reduced closed subscheme of $\bP^n_k \times_k (\bP^n_k)^{\vee}$ obtained by taking the closure of the set
\[
\big\{(x,H) \in \,\big|\, x \in H \big\}.
\]
We claim that the projection map $\beta : Z \to \bP^n_k$ is flat.  Indeed, it is clearly generically flat, and since it is finite type it is flat at some closed point.  But for any $z \in Z$, $\O_{Z, z}$ is certainly isomorphic to that of any other $z$.  Furthermore, the inclusion of rings $\O_{\bP^n_k, \beta(z)} \to \O_{Z, z}$ is also independent of the choice of $z$, up to isomorphism.  Thus $\beta$ is flat in general.

We form a commutative diagram:
\[
\xymatrix{
Y := X \times_{\bP^n_k} Z \ar[dd]_{\gamma} \ar[dr]_{\rho} \ar[rr]^{\sigma} & & \ar[dl]_{\pi} Z \ar@{^{(}->}[d] \\
& (\bP^n_k)^{\vee} & \bP^n_k \times_k (\bP^n_k)^{\vee} \ar[d] \ar[l]^-{\pi'} \\
X \ar[rr]_{\phi} & & \bP^n_k
}
\]
We describe each map appearing above.
\begin{itemize}
\item{} $\sigma$ is the projection.
\item{} $\pi'$ is the projection and thus so is $\pi$.
\item{} $\rho = \pi \circ \sigma$.
\item{} $\gamma$ is the projection.  Note that $\gamma$ is flat since it is a base change of the projection $Z \to \bP^n_k$.
\end{itemize}

\begin{exercise}
Assuming what you need to from above, show how the above diagram can be used to prove Bertini's theorem for test ideals.
\vskip 3pt
\Hint{The fibers of $\rho$ are hyperplane sections of $X$, and so one would like to show that the geometric fiber over the generic point of $(\bP^n_k)^{\vee}$ behaves in a way controllable by \autoref{quest.A2ForTau}.  One should try to use \autoref{quest.A1ForTau} to prove this, for more details see \cite{CuminoGrecoManaresiBertiniAndWeakNormality}.}
\end{exercise}

\subsection{Test ideals by finite covers and alterations}

Suppose that $(X, \Delta)$ is a pair such that $K_X + \Delta$ is $\bQ$-Cartier (of any index).  We have the following theorem

\begin{theorem}\cite{BlickleSchwedeTuckerTestAlterations} \label{thm.TestAlterations}
There exists a finite surjective separable map from a normal variety $f : Y \to X$ such that the trace map $\Tr_{Y/X} : F_* \O_Y \to \O_X$ sends $\Tr_{Y/X}( f_* \O_Y( \lceil K_Y - \pi^* (K_X + \Delta) \rceil)) = \tau(X, \Delta)$.
\end{theorem}
Indeed, one can also replace $f : Y \to X$ to be a (sufficiently large) regular alteration if you'd like $Y$ to be smooth.

Let us do a couple exercises in order to prove this result.  First we need a lemma of \cite{HunekeLyubeznikAbsoluteIntegralClosure} and generalized in \cite{SannaiSinghGaloisExtensions} which we will take as given.

\begin{lemma}
Suppose that $(R, \bm)$ is a local ring such that the Frobenius map is finite.  Suppose we are given an element $z \in H^i_{\bm}(R)$ such that the submodule of $H^i_{\bm}(R)$ generated by $\{ z, z^p, z^{p^2}, \ldots \}$ is finitely generated.  Then there exists a finite (separable) ring extension $R \subseteq S$ such that $H^i_{\bm}(R) \to H^i_{\bm S}(S)$ sends $z$ to zero.
\end{lemma}

\begin{exercise}[Hard]
Suppose that $R$ is a domain, let $N \subseteq H^{\dim R}(R)$ be the largest proper submodule of $H^{\dim R}(R)$ such that $F(N) \subseteq N$.  Show that there exists a finite extension of rings $R \subseteq S$ such that $N$ goes to 0 in $H^{\dim R}_{\bm}(R) \to H^{\dim R}_{\bm S}(S)$.
\vskip 3pt
\Hint{First show that $\ker (H^{\dim R}_{\bm}(R) \to H^{\dim R}_{\bm S}(S))$ is a submodule of $N$.  Then, using local duality, observe that $N^{\vee} = \omega_R / \tau(\omega_R)$ where $\tau(\omega_R)$ is the smallest non-zero submodule of $\omega_R$ such that $T^e(F^e_* \tau(\omega_R)) \subseteq \tau(\omega_R)$.

Suppose that $R \subseteq S$ is a finite extension of domains and $J_S = \Image(\omega_S \to \omega_R)$.  Show that there is a finite extension of domains $S \subseteq S'$ such that the support of $J_{S'} / \tau(\omega_R)$ is strictly smaller than support of $J_S / \tau(\omega_R)$.  Now proceed by Noetherian induction.}
\end{exercise}

By local duality, the above result proves the theorem in the case that $\omega_R \cong R$ and $\Delta = 0$.

To obtain the more general theorem, one has to understand the behavior of test ideals under finite maps.  We'll talk about this shortly, but first we highlight a more general question.

\begin{question}
Suppose $(X, \Delta)$ is a pair but do \emph{not} assume that $K_X + \Delta$ is $\bQ$-Cartier.  Does there exist a finite separable map (or alteration) such that $\Tr_{Y/X}( f_* \O_Y( \lceil K_Y - \pi^* (K_X + \Delta) \rceil) = \tau(X, \Delta)$?
\end{question}

This is a generalization of the biggest open question in tight closure theory (characteristic $p > 0$ commutative algebra).  Let us state this another version of this question.

\begin{exercise}
Suppose an affirmative answer to the above question.  Suppose that $X = \Spec R$ and $\tau(X, 0) = \O_X$ (in other words, that $X$ is $F$-regular).  Show that for any finite extension of $R \subseteq S$ we have that $S \cong R \oplus M$ as $R$-modules.  In other words that $R \hookrightarrow S$ splits as a map of $R$-modules.  This was proven for $\bQ$-Gorenstein rings first in \cite{SinghQGorensteinSplinters}.
\end{exercise}

Let us now introduce the other machinery needed to prove \autoref{thm.TestAlterations}.

\begin{theorem}\cite{SchwedeTuckerTestIdealFiniteMaps}
Suppose that $R \subseteq S$ is a finite separable extension of normal $F$-finite\footnote{Meaning the Frobenius map is finite.} domain.  Set $X = \Spec R$ and $y = \Spec S$.  Then for any $\bQ$-divisor $\Delta$ on $X$ we have
\[
\Tr( f_* \tau(Y, f^* \Delta - \Ram_{Y/X})) = \tau(X, \Delta).
\]
\end{theorem}

We prove the theorem in a special case.

\begin{exercise}
Prove the above theorem in the case that $\Delta \geq 0$, $(p^e - 1)(K_X + \Delta) \sim 0$ and that $f^* \Delta - \Ram_{Y/X}$ is effective.
\vskip 3pt
\Hint{There is a diagram
\[
\xymatrix{
f_* F^e_* \omega_Y \ar[r] \ar[d] & f_* \omega_Y \ar[d] \\
F^e_* \omega_X \ar[r] & \omega_X
}
\]
twist it appropriately and recall that $\Ram_{Y/X} = K_Y - f^* K_X$.}
\end{exercise}

\begin{exercise}
Prove \autoref{thm.TestAlterations} using the above methods.
\vskip 3pt
\Hint{You can use the fact that given an $\bQ$-Cartier divisor $\Gamma$ on $X$, there exists a finite separable cover $f : Y \to X$ such that $f^* \Gamma$ is Cartier.  You can also use the fact that if $D$ is Cartier, then $\tau(X, D) = \tau(X) \otimes \O_X(-D)$.}
\end{exercise}




\subsection{Other types of singularities and deformation thereof}

We introduce other types of singularities.

Note that for any variety $X$, with dualizing complex $\omega_X^{\mydot}$ we have $T : F_* \omega_X^{\mydot} \to \omega_X^{\mydot}$ dual to $\O_X \to F_* \O_X$.

\begin{definition}[$F$-injective and $F$-rational singularities]
Suppose that $X$ is a variety with dualizing complex $\omega_X^{\mydot}$.  We say that $X$ is \emph{$F$-injective} if the canonical maps $\myH^i(F_* \omega_X^{\mydot}) \to \myH^i \omega_X^{\mydot}$ surject\footnote{By local duality, this is the same as requiring that $H^i_x(\O_{X,x}) \to H^i_x(F_* \O_{X,x})$ inject for all $i$.} for all $i$.

We say that $X$ is \emph{$F$-rational} if $X$ is Cohen-Macaulay and if, locally, for every effective Cartier divisor $D$ we have that $F^e_* (\omega_X \otimes \O_X(-D)) \to \omega_X$ surjects for some $e > 0$.
\end{definition}

\begin{exercise}
Show that $X$ is $F$-rational if and only if $X$ is Cohen-Macaulay and if the only nonzero submodule $J \subseteq \omega_X$ such that $T(F_* J) \subseteq J$ is $\omega_X$. \end{exercise}

\begin{definition}
Recall that $X$ is \emph{pseudo-rational} if and only if it is Cohen-Macaulay and for any proper birational map $\pi : Y \to X$ we have that $\pi_* \omega_Y = \omega_X$.
\end{definition}

\begin{exercise}
Prove that $F$-rational singularities are pseudo-rational.
\vskip 3pt
\Hint{Use the previous exercise and show that $T(F_* \pi_* \omega_Y) \subseteq \pi_* \omega_Y$.}
\end{exercise}

\begin{theorem}
\label{thm.FInjectiveDeforms}
Suppose $(R, \bm)$ is a local ring and $f \in R$ is a non-zero divisor.  Then if $R/f$ is $F$-rational, so is $R$.  Likewise if $R/f$ is $F$-injective and Cohen-Macaulay, so is $R$.
\end{theorem}

\begin{exercise}
Prove the $F$-injective half (the second statement) of the above theorem.
\vskip 3pt
\Hint{Consider the short exact sequence $0 \to \omega_R \xrightarrow{\cdot f} \omega_R \to \omega_{R/f} \to 0$ and determine how it is effected by Frobenius (or rather $T$).  Then use Nakayama's lemma.}
\end{exercise}

\begin{conjecture}
With notation as above, if $R/f$ is $F$-injective, then $R$ is $F$-injective.
\end{conjecture}

Some recent progress on this conjecture was made by  Jun Horiuchi, Lance Edward Miller and Kazuma Shimomoto \cite{HoriuchiMillerShimomotoFInjectiveDeforms}.  $F$-injective singularities seem to correspond closely to Du~Bois singularities in characteristic zero \cite{BhattSchwedeTakagiWeakOrdinarity} and it was recently shown that Du Bois singularities satisfy the condition of \autoref{thm.FInjectiveDeforms}, see \cite{KovacsSchwedeDBDeforms}.

One can also ask the same types of questions for $F$-pure singularities.  The analog \autoref{thm.FInjectiveDeforms} for $F$-pure singularities is known to hold if $R$ is $\bQ$-Gorenstein with index of $K_R$ not divisible by $p$.  It is also known to fail without the $\bQ$-Gorenstein hypothesis \cite{FedderFPureRat,SinghDeformationOfFPurity}.  This leaves one open target

\begin{question}
Suppose that $(R, \bm)$ is a local $\bQ$-Gorenstein ring and that $f \in R$ is a non-zero divisor such that $R/f$ is $F$-pure.  Is it true that $R$ is also $F$-pure?  One can ask also for a definition of $\sigma$ for non $\bQ$-Gorenstein rings that behaves well under restriction.
\end{question}


\subsection{Bertini theorems for $F$-rational and $F$-injective singularities}

We have already explored Bertini theorems for test ideals.  There is a somewhat less ambitious target which also may be reasonable.

\begin{question}
\label{quest.BertiniFInj}
Suppose that $X \subseteq \bP^n$ is a normal and $F$-injective (respectively $F$-rational) quasi-projective variety.  Then is $X \cap H$ also $F$-injective (respectively $F$-rational) for all general hyperplanes $H \subseteq \bP^n$?
\end{question}

The problem below, analogous to \autoref{quest.A1ForTau} seems to be the missing piece needed to prove that $F$-injective and $F$-rational singularities satisfy Bertini-type theorems.

\begin{question}
 If $f : U \to V$ is a flat family with regular (but not necessarily smooth) fibers and $V$ is $F$-rational (respectively normal and $F$-injective), is it true that $U$ is also $F$-rational (respectively $F$-injective)?
\end{question}

It turns out that this is not true for $F$-injectivity without the normality hypothesis \cite[Section 4]{Enescu2009}.  Even worse, \autoref{quest.BertiniFInj} has a negative answer for $F$-injectivity without the normality hypothesis.  Let us sketch the background necessary in order to understand this failure.

\begin{definition}
Suppose that $X = \Spec R$ is a variety over a field of characteristic $p > 0$.  We say that $X$ is \emph{weakly normal} if $z \in K(R)$ and $z^p \in R$ implies that $z \in R$.  We say that $X$ is \emph{\textnormal{WN1}} if it is weakly normal and the normalization morphism $\Spec S = X^{\textnormal{N}} \to X = \Spec R$ is unramified in codimension one.\footnote{Recall that an extension of local rings $(R, \bm) \subseteq (S, \bn)$ is \emph{unramified} if $\bm \cdot S = \bn$ and $R/\bm \subseteq S/\bn$ is separable.}
\end{definition}

In \cite{CuminoGrecoManaresiBertiniAndWeakNormality} the authors showed that if weakly normal surfaces that are not WN1 have general hyperplanes that are \emph{not} weakly normal.  If you combine this with the fact that $F$-injective singularities are weakly normal \cite{SchwedeFInjectiveAreDuBois}, all we must do in order to construct an $F$-injective quasi-projective variety whose general hyperplane is not $F$-injective is to find an $F$-injective but not WN1 surface.

\begin{exercise}
Suppose $k$ is an algebraically closed field of characteristic $p > 0$.  Consider $R = k[x,y]$ and let $I = \langle y (y-1) \rangle$ so that $V(I)$ is two lines.  Let $Y = \Spec k[t]$ and let $V(I) \to Y$ be the map which is the identity on one component but Frobenius on the other.  Let $S$ to be the pullback of the diagram of rings $\{ R \to R/I \leftarrow k[t] \}$ with maps as described above.  In other words, $\Spec S$ is the pushout of $\{ X\leftarrow V(I) \to Y \}$.

Show that $\Spec S$ is $F$-injective but not WN1.
\vskip 3pt
\Hint{To show that $\Spec S$ is not WN1 note that $X$ is its normalization.  To show that it is $F$-injective, analyze the short exact sequence $0 \to S \to R \oplus k[t] \xrightarrow{-} R/I \to 0$.\footnote{Here, the map denoted $-$ is the difference of the maps on each of the two components.}}
\end{exercise}

\bibliographystyle{skalpha}
\bibliography{master}

\end{document}